\documentclass[12pt]{amsart}
\usepackage[T1]{fontenc}
\usepackage[utf8]{inputenc}
\usepackage[french,english]{babel}
\usepackage{graphicx}
\usepackage{refcount}
\usepackage{color}
\usepackage{enumitem}
\setenumerate[1]{label={\arabic*.}}
\usepackage{subfigure}

\newtheorem{prop}{Proposition}[subsection]
\newtheorem{theo}[prop]{Theorem}
\newtheorem{theo2}[prop]{Classification Theorem}
\newtheorem{cor}[prop]{Corollary}
\newtheorem{lem}[prop]{Lemma}
\theoremstyle{definition}
\newtheorem{deff}[prop]{Definition}
\theoremstyle{remark}
\newtheorem{rem}[prop]{Remark}
\newtheorem{rems}[prop]{Remarks}
\newtheorem*{question}{Questions}
\newtheorem{Q}[prop]{Question}

\newcommand\Pp{\mathbb P}

\catcode`@=11

\def\recopier#1{#1}
\def\@zero{0}
\def\sanspt#1{\edef\a{#1}\ifx\a\@zero Faux\else\edef\a{\expandafter\recopier\a}%
  \edef\a{\expandafter\recopier\a}\expandafter\Sanspt\a\fi}
\def\Sanspt#1.{#1}

\catcode`@=12

\let\joli=\mathcal 
\newcommand\ftilde{\tilde f}
\newcommand\gtilde{\tilde g}
\newcommand\Cp{\mathbb{C}} \def\R{\mathbb{R}}
\newcommand\N{\mathbb{N}} \def\Z{\mathbb{Z}}
\newcommand\Ss{\mathbb S}
\newcommand\inv{^{-1}}

\newcommand\vt{v^{\scriptscriptstyle{1\over 2}}}
\def\vtt^#1{v^{\scriptscriptstyle #1}}
\newcommand\Tt{{T_{\scriptscriptstyle{1\over 2}}}}
\newcommand\mTt{T_{\scriptscriptstyle{-{1\over 2}}}}
\newcommand\vun{v^{\scriptscriptstyle1}}
\newcommand\zpoint{{\dot z}}

\newcommand\hchapeau{{\widehat h}}
\newcommand\Ftilde{\widetilde F}%
\newcommand\fdag{{f^{\dagger}}}
\newcommand\gdag{{g^{\dagger}}}
\newcommand\zbar{{\overline z}}

\newcommand\wbar{{\overline w}}
\newcommand\sgn{\mathop{\rm sgn}\nolimits}

\newcommand\Diff{\mathop{\rm Diff}\nolimits}%
\newcommand\DiffC{\mathop{\rm Diff}\nolimits_1(0,\Cp,\overline\Cp)}%
\def\AA_#1{\joli A_{#1}\cup \overline{\joli A_{#1}}}%
\newcommand\AAkb{\AA_{k,b}}%
\let\phi=\varphi

\let\del=\partial
\newcommand\STt{\Sigma\circ\Tt}
\def\STtt_#1{\Sigma\circ T_{\scriptscriptstyle#1}}
\def\Ttt_#1{{T_{\scriptscriptstyle{#1}}}}
\expandafter\let\expandafter\over\csname @@over\endcsname
\expandafter\let\expandafter\atop\csname @@atop\endcsname

\let\wt=\widetilde

\def\raggedcenter{\leftskip=0pt plus4em \rightskip=\leftskip%
  \parfillskip=0pt \spaceskip=.3333em \xspaceskip=.5em
  \pretolerance=9999 \tolerance=9999 \parindent=0pt
  \hyphenpenalty=9999 \exhyphenpenalty=9999 }

\font\bfmi=cmmib10 at 12truept
\title[Germs of antiholomorphic parabolic diffeomorphisms]
{Analytic classification of germs of 
  parabolic antiholomorphic diffeomorphisms of codimension k}
\author{Jonathan Godin}
\email{godinj@dms.umontreal.ca}
\author{Christiane Rousseau}
\email{rousseac@dms.umontreal.ca}
\address{Universit\'e de Montr\'eal, C.P.\null{} 6128, Succ.\null{} Centre-Ville,
  Montr\'eal, Qc, Canada, H3C 3J7}
\date{\today}
\thanks{The first author was supported by a FRQ-NT PhD scholarship. 
  The second author is supported by NSERC in Canada.}

\usepackage{caption}
\usepackage[hidelinks]{hyperref}
\hypersetup{
  pdftitle={Analytic Classification of Germs of Parabolic Antiholomorphic 
    Diffeomorphism of Codimension k},
  pdfauthor={Jonathan Godin, Christiane Rousseau},
  pdfsubject={Complex dynamical systems},
  pdfkeywords={Antiholomorphic Dynamics, Antiholomorphic Parabolic Fixed Point, Local
  Classification under Analytic Conjugation, Space of Orbits near a Fixed Point}
}

\begin{document}

\begin{abstract}
  We investigate the local dynamics of antiholomorphic
  diffeomorphisms around a parabolic fixed point. We first 
  give a normal form. Then we give a complete classification 
  including a modulus space for antiholomorphic germs with 
  a parabolic fixed point under analytic conjugacy. We then 
  study some geometric applications: existence of real analytic 
  invariant curve, existence of holomorphic and antiholomorphic 
  roots of holomorphic and antiholomorphic parabolic germs, 
  commuting holomorphic and antiholomorphic parabolic germs.

\end{abstract}

\maketitle
\tableofcontents

\section{Introduction}

In this paper, we are interested in the local
dynamics of antiholomorphic diffeomorphisms with a
parabolic fixed point, i.e.\null{} a fixed point of
multiplicity $k+1$ (i.e.\null{} of codimension $k$). 
We study the classification under conjugacy by analytic
changes of coordinate of a germ of an anti\-holomorphic 
diffeomorphism $f$ with a parabolic fixed point.
In local coordinate, it may be chosen in the form
\begin{equation}\label{eq:f 1er}
  \textstyle
  f(z) = \zbar + {1\over2} \zbar^{k+1} + o(\zbar^{k+1})
\end{equation}
for some integer $k \geq 1$. 

The classification of parabolic fixed points
in the holomorphic case for
a germ 
\begin{equation}
  g(z) = z + z^{k+1} + \left({k+1\over 2}-b\right)z^{2k+1} + o(z^{2k+1})
\end{equation}
is well known. (See e.g.\hbox{} \cite{nonlinear} 
or \cite{LectDiffEq}.) The dynamics of $g$ 
(see Figure~\ref{fig:dyn g}) is determined by 
a topological invariant, the integer $k$,
a formal invariant, the complex number $b$,
and an analytic invariant given by an equivalence class
of $2k$ germs of diffeomorphisms which
are the transition functions on the space
of orbits of $g$ (the Écalle horn maps). Two germs $g_1$ and $g_2$
are formally equivalent if and only if they have the
same topological invariant and formal invariant.
Furthermore, they are analytically equivalent
if and only if they also have the same analytic
invariant.

\begin{figure}[htbp]
  \centering
  \includegraphics{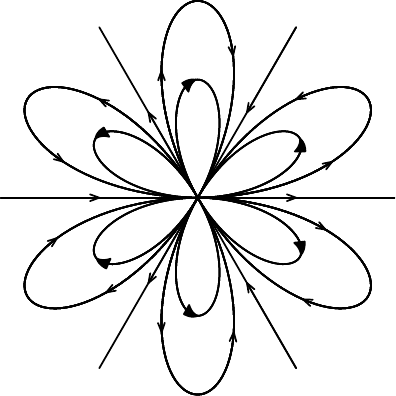}
  \caption{Dynamics of a holomorphic parabolic germ
  with topological invariant $k=3$}
  \label{fig:dyn g}
\end{figure}

The goal of this paper is to establish a 
local classification of anti\-holomorphic parabolic germs
under the analytic conjugation and to
describe the space of orbits of such a germ
and, more generally, to explore the geometric properties
of antiholomorphic parabolic germs which are invariant
under analytic conjugation.
This is done for fixed points of any multiplicity.
It allows us to provide a solution to the following problems.

\begin{question}
  \begin{enumerate}
    \item \label{Q:racine}\label{Q:premier}
      \textit{(Antiholomorphic Root Extraction)}
      The second iterate of an antiholomorphic parabolic
      germ $f$ as in equation~\eqref{eq:f 1er}
      is a holomorphic germ which is parabolic. When 
      is the converse true: Given a parabolic germ of holomorphic
      diffeomorphism $g$, when is it possible to write
      it as $g = f\circ f$, for some antiholomorphic parabolic germ
      $f$? We call $f$ an \emph{antiholomorphic square root}
      of $g$. More generally, when does $g$ have an antiholomorphic
      root of some order? 
    \item \label{Q:racine antihol}\ignorespaces\thinspace 
      Analogously, when does an antiholomorphic parabolic germ have
      an antiholomorphic root? When are the roots
      unique?
    \item\label{Q:flot}
      \textit{(Embedding)} Let $\{v^t\}_t$, where $v^t\colon z\mapsto v^t(z)$,
      be the flow of the differential equation $\dot z = v(z) = {z^{k+1}\over 1+bz^k}$.
      Each $v^t$ ($t\not=0$) is a holomorphic germ with a parabolic fixed
      point at the origin. Then $\overline{\vt(\cdot)}$ is
      an antiholomorphic germ, and any antiholomorphic parabolic
      germ is formally conjugate to such a germ. Given an antiholomorphic
      parabolic germ $f$, when is it analytically conjugate 
      to some $\overline{\vt(\cdot)}$? In that case, it allows
      to embed $f$ in the familly $\overline{v^t}$.
    \item \label{Q:courbe reelle} 
      When does an antiholomorphic germ preserve a germ of real analytic curve? 
      This is equivalent to say that the germ is analytically conjugate 
      to a germ with real coefficients.
    \item \label{Q:central} 
      \textit{(Centralizer)} 
      Can we describe all the antiholomorphic parabolic germs
      $f$ that commute with a holomorphic parabolic germ $g$? 
      If $f$ and $g$ commute, then $f$ sends the orbits of 
      $g$ on the orbits of $g$. This greatly restricts 
      the possible $f$. In an analogous way, can we describe 
      all the holomorphic and the antiholomorphic germs
      that commute with an antiholomorphic parabolic germ?
    \label{Q:dernier}
  \end{enumerate}
\end{question}

The above problems are questions about the equivalence
classes of germs under analytic conjugacy.
Therefore, the answer should be read in the
modulus of classification, which will be introduced
in Section~\ref{sec:module}.

The local dynamics of an antiholomorphic parabolic germ has
similarities with the holomorphic case: indeed
the $n$-th iterate $f^{\circ n}$ is holomorphic 
for $n$ even. We find that the dynamics is determined by the same
topological and formal invariants, but the
analytic invariant is composed of $k$ germs of
diffeomorphisms, instead of $2k$. This is explained
by the fact that an orbit of $f$ will usually jump between
two Fatou petals of its associated holomorphic parabolic germ $f^{\circ 2}$
(see Figure~\ref{fig:orbit haut bas}), so that the
dynamics in those petals are not independent.

\begin{figure}[htb]
  \includegraphics{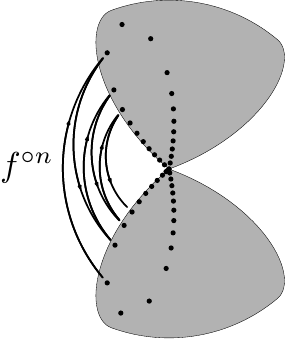}
  \caption{An orbit of $f$ jumping between
  two petals. An orbit of the second iterate $f^{\circ 2}$ will
  remain either in the upper petal or the
  lower petal.}
  \label{fig:orbit haut bas}
\end{figure}

We observe other differences from the holomorphic case.
A holomorphic germ has $2k$ formal separatrices. The
antiholomorphic germ has instead a privileged unique 
direction; a \textit{formal symmetry axis}. 
There is also a topological difference between 
the cases where the codimension is odd or even. When $k$ is 
even, the rotation $z\mapsto -z$ is a formal symmetry 
of $f$, whilst it is not for $k$ odd.


Antiholomorphic dynamics has been considered before in the context
of anti-polynomials, that is a polynomial function
of $\zbar$, $p(z) = \zbar^n + \cdots + a_0$. Iteration of anti-polynomials
was studied by Nakane and Schleicher in \cite{multiI},
Hubbard and Schleicher in \cite{multiNotPath}, and Mukherjee,
Nakane and Schleicher in \cite{multiII}.
Their focus is mostly on the family of anti-polynomials
$p_c(z) = \zbar^d + c$ and the description
of the connectedness locus $\mathcal M_d^\ast$
called the \emph{multicorn}. The context is
global in nature, but the local analysis contributes
significantly.

An important role is played by periodic orbits
of $p_c$ of odd period $k$, because when $k$
is odd, $p_c^{\circ k}$ is antiholomorphic. In this
case, all indifferent periodic orbits are parabolic
and they occur along real analytic arcs in the
parameter space, as proved in 
\cite{multiNotPath} and \cite{multiII}. 
However, only points of codimension~1 and~2 are observed. 
This is due to a choice of 
a special subfamily of anti-polynomials of degree
$d$. Indeed, higher codimension is already observed
in the 2-parameter family $\zbar^d + c_1\zbar + c_0$,
e.g.\null{} when $c_1 = 1$ and $c_0=0$.

One of the tools used for anti-polynomials is
called the \emph{Écalle height}, introduced
by Hubbard and Schleicher in \cite{multiNotPath}.
In codimension~1, on the Écalle cylinder of the
attractive petal, the imaginary part of an orbit
is intrinsic, and this is used to prove that
landing points of curves in the parameter space where
$p_c$ has a parabolic periodic orbit of odd period is
are points of codimension~2. When studying the space of
orbits of an antiholomorphic germ of parabolic diffeomorphism
of any codimension, we see that the Écalle height
has a meaning only on the Écalle cylinder of the petals
containing the formal symmetry axis of $f$.
This is seen by describing the space of orbits on a 
neighbourhood of a parabolic fixed point, which
we do in Section~\ref{sec:espace des orbites f}.

The paper is organized as follows.
In Section~\ref{sec:pt fixe parabolique}, we define the topological
and formal invariants of $f$. We also establish
a formal normal form for $f$.

In Section~\ref{sec:normal}, 
we study the formal normal
form. 

In Section~\ref{sec:Fatou}, we introduce the
rectifying coordinate and the Fatou coordinates
in order to define
the transition functions (Def{.}~\ref{def:module})
in Section~\ref{sec:module}, which is the
analytic invariant. This leads to the modulus 
of classification of $f$.

In Section~\ref{sec:espace des orbites}, we
recall a description of the space of orbits
in the holomorphic case using $2k$ spheres
(or Écalle cylinders) glued with the horn
maps (these are the expressions of the transition
functions in the coordinates of the spheres). 
We use this space of orbits in
Section~\ref{sec:espace des orbites f}
to identify the space of orbits of $f$ 
to a manifold of real dimension~2 
by quotienting the space of orbits of $f\circ f$
by the action of $f$. 

After describing the space of orbits,
we state, in Section~\ref{sec:class theo},
the main result of the paper: the Classification Theorem~\ref{theo:class}. 
The idea in spirit is that two germs are equivalent if
and only if their space of orbits are equivalent;
the Classification Theorem is a way to rigorously express
this statement.

Finally, with the Classification Theorem in hands,
in Section~\ref{sec:applications} we answer Questions~%
\sanspt{\getrefnumber{Q:premier}} to~\sanspt{\getrefnumber{Q:dernier}}
of the introduction.

\section{Antiholomorphic Parabolic Fixed Points}\label{sec:pt fixe parabolique}
\noindent{\bf Notation.} For the whole paper,
we will use the following notation~:
\begin{itemize}
  \item $\sigma(z) = \zbar$ is the complex conjugation;
  \item $\tau(w) = {1 \over \overline w}$ is the antiholomorphic inversion;
  \item $T_C(Z) = Z + C$ is the translation by $C\in \Cp$;
  \item $L_c(w) = cw$ is the linear transformation with multiplier $c\in \Cp$;
  \item $v^t$ is the time-$t$ of the vector field
    \begin{equation}\zpoint = v(z) = {z^{k+1} \over 1 + bz^k}.\end{equation}
\end{itemize}
\bigbreak

A function $f\colon U \to \Cp$ defined on a domain
$U\subseteq \Cp$ is antiholomorphic if
${\del f\over \del z} \equiv 0$ on $U$.
From this definition, together with the chain rule,
it follows that antiholomorphy is an intrinsic 
property of $f$ under holomorphic changes of variable.
Equivalently, $f\colon z\mapsto f(z)$ is antiholomorphic
if $f\circ\sigma\colon z\mapsto f(\zbar)$ is holomorphic, therefore
$f(z)$ expands in a power series in terms of $\zbar$.

Note that the multiplier at a fixed point
of an antiholomorphic function is not intrinsic, only its modulus
is. Indeed, a scaling of $\lambda$ will add a factor
of $\lambda\over \overline\lambda$ to the multiplier.

\vbox\bgroup
\begin{deff}[Parabolic fixed point]
A germ of antiholomorphic diffeomorphism
fixing the origin $f\colon (\Cp,0)\to (\Cp,0)$ 
has a \emph{parabolic fixed point} at $0$ if 
$0$ is an isolated fixed point and
$$
    \left|{\del f\over \del \zbar}(0)\right| = 1.
$$
We will also say that $f$ is an \emph{antiholomorphic
parabolic germ}.
\end{deff}
\egroup

\begin{prop}\label{prop:forme reelle}
  Let $f(z) = a_1\zbar + a_2 \zbar^2 + a_3 \zbar^3 + \cdots$.
  If $|a_1| = 1$, then $f$ is formally conjugate to
  a formal power series 
  \begin{equation}\label{eq:serie reel}
    \fdag(w) = \wbar + \sum_{n=2}^\infty A_n \wbar^n
  \end{equation}
  with real coefficients $A_n$. If there exists $n \geq 2$ such
  that $A_n\not =0$, then $0$ is a parabolic fixed point of $f$. 
  Let $n_0 =k+1$ be the minimum such $n$.
  Then a scaling brings $A_{k+1}$
  to ${1\over 2}$ if $k$ is odd
  (resp.\null{} $\pm{1\over 2}$ if
  $k$ is even).
\end{prop}
\begin{proof}
  The proof is a mere computation. Let
  $w = \hchapeau(z) = \sum_{n\geq 1} b_n z^n$ be a formal change of
  coordinate and suppose $\fdag(w) = \wbar + \sum_{n\geq 2} A_n \wbar^n$.
  If we compare $h\circ f(z) = \fdag \circ h(z)$ degree by degree,
  we find an expression for the coefficients of the form
  $$
  \begin{cases}
    b_1a_1 = \overline{b_1}\\[2pt]
    A_n = b_n - \overline{b_n} + a_n + P_n(A_1,\ldots, A_{n-1},
    a_1,\ldots,a_{n-1}, b_1,\ldots, b_{n-1}),
  \end{cases}
  $$
  where $P_n$ is some polynomial. Hence, we have 
  $\arg b_1 = -{1\over2} \arg a_1 + \ell \pi$, with $\ell\in\Z$.
  With a recursive argument,
  if $A_1,\ldots, A_{n-1}$ are real, for $A_n$ to be real,
  we may choose $\Im b_n = {1\over 2}\Im(a_n + P_n)$,
  since $P_n$ depends only on terms that were fixed in the
  previous steps.
\end{proof}

\begin{rem}\label{rem:changement formel}
  The formal change of coordinate $\hchapeau$ is not unique.
  Indeed, only the imaginary part of the coefficients are
  determined, leaving their real part free. However, the order
  of the first non linear term is well defined. This leads
  to the following definition.
\end{rem}
\begin{deff}\label{def:codim}
  We say that $f$ is parabolic \emph{of codimension $k$} if the first non linear
  term of $\fdag$ is of order $k+1$.
\end{deff}
\begin{rems}
  \begin{enumerate}
    \item The formal series with real coefficients preserves the
  real axis. This indicates that $f$ has a privileged unique
  direction which we will call a \emph{formal 
  symmetry axis}. Hence a conjugacy between two 
  antiholomorphic parabolic germs must preserve the formal
  symmetry axis. We can of course suppose that this formal 
  axis is the real axis. Note however that in the case of 
  even codimension, there is no canonical orientation of the 
  formal symmetry axis. 
  
  \item The dynamics nearby the formal symmetry axis 
  is a topological invariant. When $k$ is odd, a rotation 
  of angle $\pi$ will flip the attractive semi-axis with 
  the repulsive one. When $k$ is even, both semi-axes are 
  either attractive (when $A_{k+1} < 0$) or repulsive 
  (when $A_{k+1} > 0$) (see Figures~\ref{fig:secteurs inv 
  impair} and~\ref{fig:secteurs inv pair}). In this paper, 
  {\bf we will only consider the case $A_{k+1} > 0$}. Indeed,
  when $A_{k+1} < 0 $, i.e.\null{} $f$ is of \emph{negative type},
  then $f\inv$ will be of positive type 
  and classifying $f\inv$ is equivalent to classifying $f$.
  \end{enumerate}
\end{rems}

\begin{deff}\label{def:type +} When the codimension $k$ is even,
  we say $f$ is of \emph{positive type} (resp.\null{}
  \emph{negative type}) if $A_{k+1} > 0$ (resp.\null{}
  $A_{k+1} < 0$), where $A_{k+1}$ is the first non zero 
  coefficient in~\eqref{eq:serie reel}.
\end{deff}

\begin{figure}[htbp]
  \centering
  \includegraphics{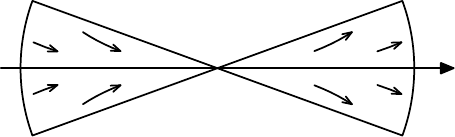}\hfil
  \includegraphics{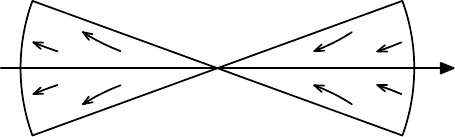}
  \caption{Dynamics near the formal symmetry axis
    of $f(z) = \zbar + o(\zbar)$ in \emph{odd}
    codimension. One sector is attractive and the
    other repulsive; this yields the two possibilities
    above.}
  \label{fig:secteurs inv impair}
\end{figure}
\begin{figure}[htbp]
  \centering
  \includegraphics{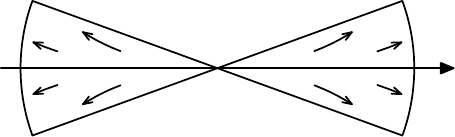}\hfil
  \includegraphics{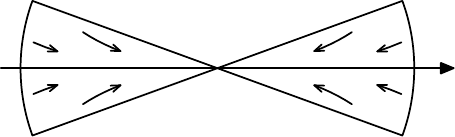}
  \caption{Dynamics near the formal symmetry axis
    of $f(z) = \zbar + o(\zbar)$ in \emph{even}
    codimension. The possibilities are: on the
    left, both sectors are repulsive (positive 
    type) and, on the right, both are attractive
    (negative type).}
  \label{fig:secteurs inv pair}
\end{figure}

The composition of two antiholomorphic germs
is holomorphic. Therefore, we will look at 
$g:= f\circ f$, which is a holomorphic
parabolic germ. Recall that in
the holomorphic case, the codimension
of $g$ is the order of the first non zero term
of $g(z) - z$. It is linked to the multiplicity of the fixed point:
$g$ is of codimension $k$ if and only if the fixed
point has multiplicity $k+1$. 

\begin{cor}
  $f$ is of codimension $k$ if only if $g = f\circ f$
  is of codimension $k$.
\end{cor}

The case when $f \circ f = id$ is seen as a
degenerate case where $f$ is of ``codimension infinity''.
Indeed, it only happens if $f$ is analytically conjugate to
the complex conjugation, as is shown below. 
This case was excluded from our definition of 
parabolic point, since the fixed point of $\sigma$ at the origin
is not isolated.

\begin{prop}
  Let $f(z) = a_1\zbar + a_2 \zbar^2 + a_3 \zbar^3 + \cdots$
  be an antiholomorphic germ at the origin.
  The following statements are equivalent:
  \begin{enumerate}
    \item $f$ is formally conjugate to $\sigma$;
    \item $f$ is analytically conjugate to $\sigma$;
    \item $f\circ f = id$.
  \end{enumerate}
\end{prop}
\begin{proof}
  $1.\Rightarrow 3.$
  Since there is a formal change of
  coordinate $m$ such that $m\circ f\circ m\inv = \sigma$, 
  we have $m \circ f\circ f\circ m\inv = id$ formally, which yields
  $f\circ f = id$.

  \noindent $3.\Rightarrow 2.$
  Let us suppose that $f\circ f=id$. In particular, $|a_1| = 1$.
  We can of course suppose $f$ is already in a coordinate such that $a_1 = 1$.

  Let $F_1(x,y) = \Re f(z)$ and $F_2(x,y) = \Im f(z)$, then we
  have
  $$
    F(x,y) = \begin{pmatrix} F_1(x,y)\cr F_2(x,y) \end{pmatrix}
      = \begin{pmatrix} 
	x + O(|(x,y)|^2)\cr -y + O(|(x,y)|^2)\cr\end{pmatrix}.
  $$
  We are interested in the fixed points of $F$, that is the zeros
  of $F-id$. Since ${\del\over \del y}(F_2(x,y) - y)\big|_{(0,0)} = -2$,
  by the Implicit Function Theorem, there exists an analytic
  curve $\gamma\colon t\mapsto t + i\eta(t)$ such that $F_2(x,y) - y = 0$ if and only if
  $y = \eta(x)$.
  
  We complexify $t$ to obtain a change of coordinate
  $t = \gamma\inv(z) = u + iv$ that rectifies the curve $\gamma$ on the real line.
  Let $\ftilde = \gamma\inv\circ f\circ \gamma$. In the new coordinate, 
  $\Ftilde = \gamma\inv\circ F\circ \gamma$ has now the form
  $$
    \Ftilde(u,v) = \begin{pmatrix}
      u + r(u,v) \cr -v(1 + O(|(u,v)|^2))
    \end{pmatrix},
  $$
  where $r(u,v) = O(|(u,v)|^2)$. The equation for fixed
  points $\Ftilde = id$ is equivalent to $v=0$ and
  $r(u,0) = 0$. If $r(u,0) = au^s + o(u^s)$, $a\not=0$, then this would
  contradict the fact that we must have
  $\Ftilde\circ \Ftilde(u,0) = \left({u\atop 0}\right)$.
  Therefore $r(u,0) \equiv 0$, in other words 
  $r(u,v) = v p(u,v)$.

  We see that the real axis is a line of fixed points for $\ftilde$ near the origin.
  By the Identity Theorem, because
  $\ftilde\circ \sigma - id = 0$ on the real axis near the origin, we have
  $\ftilde \equiv \sigma$.

  \noindent $2.\Rightarrow 1.$ This is immediate.
%
%
\end{proof}

\color{black}

The formal power series with real coefficients is used
to determine a formal normal form for $f$. Recall
that a formal normal form for $g:=f \circ f$ may be 
taken as the time-1 map of the
flow of~(see~\cite{LectDiffEq})
\begin{equation}\label{eq:champ}
  \zpoint = {z^{k+1} \over 1 + bz^k}
\end{equation}
for some constant $b\in \Cp$. We will call this constant $b$
the \emph{formal invariant}. It is also sometimes called
the ``résidu itératif'' and, as mentioned in \cite{multiNotPath},
it is determined by the holomorphic fixed point index, that is, the 
residue of ${1 \over z - g(z)}$ at the origin.

When $g = f\circ f$ is of codimension $k$, it is
possible to get rid of the terms of degree
$k+1 < j < 2k+1$ by an analytic change of
coordinate. In this coordinate, $g$ is
written
\begin{equation}\label{eq:g prenormal}
  g(z) = z + z^{k+1} + \left( {k+1 \over 2} - b\right)z^{2k +1} 
    + o(z^{2k + 1}),
\end{equation}
where $b\in\Cp$ is the formal invariant of $g$.
When $g$ is in this form, we will say that it is \emph{prenormalized}.

\begin{deff}\label{def:formel}
  The \emph{formal invariant} of $f$ is the 
  formal invariant of $f \circ f$, which
  is the constant $b$ in~\eqref{eq:g prenormal}.
\end{deff}

As the name suggests, $b$ is invariant under formal
changes of coordinate. Since 
$\gdag := \fdag \circ \fdag$ and $g$ have the
same formal invariant, where $\fdag$ is as 
in~\eqref{eq:serie reel}, it follows that $b$ is
real because all the coefficients of $\gdag$
are real.

An important consequence of this, is that the
time-$t$ map $v^t$ of~\eqref{eq:champ} 
for $t\in\R$ has a power series at $0$ with real coefficients,
that is the complex conjugation $\sigma$ and $v^t$ commute.
    
\begin{prop}
  Let $\vt$ be the time-${1\over2}$ of the vector field~\eqref{eq:champ}
  for some $b$. If $f$ is of codimension~$k$, of positive type if $k$
  is even and if $f$ has formal invariant $b$,
  then $f$ and $\sigma\circ\vt$ are formally conjugate.
\end{prop}
\begin{proof}
  Let $\fdag$ be the formal power series with real
  coefficients formally conjugate to $f$ in
  Proposition~\ref{prop:forme reelle}. Then $\fdag\circ \sigma$
  is a parabolic formal power series of $z$.
  A formal normal form may be chosen as $\vt$, the time-$1\over 2$
  of the vector field~\eqref{eq:champ}. Since both the coefficients
  of $\vt$ and $\fdag$ are real, the formal change of coordinate $h$
  commutes with $\sigma$, provided that $h'(0) = 1$ , 
  so that $\vt\circ \sigma$ is formally conjugate to $f$.
\end{proof}

The formal change of coordinate conjugating $f$
to its formal normal form can always be truncated
at the $(2k+2)$-th term, which yields a holomorphic change
of coordinate taking $f$ to the form
\begin{equation}\label{eq:prenormal}
  f(z) = \zbar + {1\over 2}\zbar^{k+1} + \left({k+1 \over 8}
    - {b \over 2}\right)\zbar^{2k+1} + o(\zbar^{2k+1}),
\end{equation}
that is $f$ and $\sigma\circ \vt$ have the same
first three terms.
\begin{deff}\label{deff:prenormal}
  When $f$ is in the form~\eqref{eq:prenormal}, we will say that it
  is \emph{prenormalized}.
\end{deff}

\begin{rem}
  In even codimension, $f$ may only be prenormalized as
  in~\eqref{eq:prenormal} when $A_{k+1} > 0$. In odd codimension,
  $f$ may always be prenormalized as in~\eqref{eq:prenormal}.
\end{rem}

The formal normal form is a model to which the germs can
be compared. Now that this form has been established, we
describe its properties.

\section{Properties of the Formal Normal Form}\label{sec:normal}
Let us start with the following observations.
\begin{prop}\label{prop:prop v} Let $v$ be the vector field~\eqref{eq:champ}
  of codimension $k$ and formal invariant $b$.
  \begin{enumerate}
    \item $v$ is invariant under the rotations of order $k$.
    \item $v$ is invariant under the complex conjugation $\sigma$
    when $b$ is real.
  \end{enumerate}
\end{prop}

The holomorphic and antiholomorphic formal normal
forms are respectively
\begin{align}\label{eq:vun}
  v^1(z) &= z + z^{k+1} + \left({k+1\over 2} - b\right)z^{2k+1}
    + o(z^{2k+1})\\[5pt]\label{eq:vt}
  \sigma\circ \vt(z) &= \zbar + {1\over2}\zbar^{k+1} 
    + \left({k+1\over 8} - {b\over 2}\right)\zbar^{2k+1}
    + o(\zbar^{2k+1}),
\end{align}
where $v^t$ is the time-$t$ of $v$.

We see that the real axis is a symmetry axis. We introduce
a notation for the other symmetry axes.
\begin{deff}
  Let $\sigma_\ell$ denote the reflection
  \begin{equation}
    \sigma_\ell(z) := e^{2i\pi\ell\over k} \zbar,
    \qquad\rlap{\qquad for $\ell=0,\ldots,k-1$.}
  \end{equation}
\end{deff}

\begin{cor}\label{cor:prop v}
  \null\hfill
  \begin{enumerate}
    \item $v$ is invariant under $\sigma_\ell$ for $\ell=0,\ldots,k-1$ when
      $b$ is real;
    \item $\vun$ commutes with any rotation of order $k$, and when $b$ is real,
      it commutes with $\sigma_\ell$ for $\ell = 0,\ldots,k-1$;
    \item When $k$ is even, $\sigma\circ \vt$ commutes with $z\mapsto -z$.
  \end{enumerate}
\end{cor}

We will only be interested in real values of $b$.

\vbox{
\begin{prop}[Roots of the normal forms]\label{prop:racine v}
  \null\hfill
  \begin{enumerate}
    \item For $n$ even, $\vun$ has $k$ one-parameter families
    of antiholomorphic $n$-th roots given by $\sigma_\ell\circ \vtt^{{1\over n} + iy}$
    for $y\in\R$, $\ell=0,\ldots,k-1$.
    \item For $n$ odd, $\sigma\circ \vt$ has exactly one antiholomorphic
    $n$-th root given by $\sigma\circ \vtt^{1\over 2n}$.
  \end{enumerate}
\end{prop}
}

We ask the following questions,
which will be answered in Section~\ref{sec:racine}.
\begin{Q}\label{Q:racines bis}
  For a holomorphic parabolic germ $g$, how many distinct
  antiholomorphic $n$-th roots ($n$ even) does it have?
\end{Q}

\begin{Q}\label{Q:racines antihol}
  For an antiholomorphic parabolic germ $f$ and $n$ odd, when is the formal
  $n$-th root convergent?
\end{Q}

\section{Fatou Coordinates}\label{sec:Fatou}
For the whole section, when the codimension of $f$
is even, we will suppose $f$ is of positive type
(see Definition~\ref{def:type +}).
The formal normal form $\sigma\circ \vt$ is a model
to which it is natural to compare the antiholomorphic germ
$f$. In the holomorphic study of parabolic germs, 
we use holomorphic diffeomorphisms
called Fatou coordinates
defined on sectors covering the origin on which
the germ is conjugated to its normal form, i.e.\null{} 
changes of coordinates to the normal form. We
then compare Fatou coordinates on the intersection
of the sectors, thus yielding a conformal invariant 
describing the space of orbits of the germ.
See \cite{nonlinear} or \cite{LectDiffEq}
for the details.

The same approach can be adapted to the
antiholomorphic case. It will be necessary
to find a sectorial normalization (Fatou
coordinates) of the antiholomorphic germ
$f$. However, instead of adapting the construction
of the holomorphic case, we will prove that
it is possible to choose Fatou coordinates
of $f \circ f$, which is holomorphic, that
are also Fatou coordinates of $f$.

\subsection{Rectifying Coordinates and sectors} 
Suppose that an antiholomorphic parabolic germ $f$ is of codimension~$k$ for
$k\geq 1$ with a formal invariant $b$ (see Def.\null{} \ref{def:formel}).
The Fatou coordinates $\varphi_j$ are often constructed
in the rectifying coordinate given by the time of the
vector field~\eqref{eq:champ}. Since $\vt$ and $\vun$
are the time maps of the vector field~\eqref{eq:champ},
we define the time coordinate by
\begin{equation}\label{eq:coord temps}
  Z(z) = \int_{z_0}^z {1 + b\zeta^k \over \zeta^{k+1}}{\rm d}\zeta
    = -{ 1\over kz^k} + b\log z + {1 \over kz_0^k} - b\log z_0
\end{equation}
which is multi-valued. See Figure~\ref{fig:stationnement}
for its Riemann surface. It is the inverse of the flow
of~\eqref{eq:champ} with starting point $z_0$. We will
single out the following $2k+1$ charts of $Z$:
\begin{equation}\label{eq:Zj}
\setbox0=\hbox{$Z_j(z)$}
\dimen0=\wd0
\setbox0=\hbox{$ {}=-{1 \over kz^k} + b\log z + {ji\pi \over k}$}
\dimen1=\wd0
\dimen2=\textwidth
\advance\dimen2 by -\dimen0
\advance\dimen2 by -\dimen1
\dimen2=.5\dimen2
\advance\dimen2 by \dimen1
  \begin{aligned}
    Z_j(z) &= -{1 \over kz^k} + b\log z - {ji\pi b\over k}\\[2mm]
      & \setbox0\hbox{for $j=-k,\ldots,-1,0,1,\ldots,k$,}
        \rlap{\kern\dimen2\kern-\wd0\rlap{\unhbox0}}
  \end{aligned}
\end{equation}
where $\log z$ is determined by $\arg z\in(-\pi,\pi)$
for $-k < j < k$, and for $Z_k$ (resp.\null{} $Z_{-k}$),
$\arg(\cdot)$ will be the continuation in
$(0,2\pi)$ (resp.\null{} in $(-2\pi,0)$).
In particular, we see that $Z_k = Z_{-k}$, and that
\emph{both $Z_0$ and $Z_k$ commute with the complex
conjugation.} 

\begin{figure}[htb]
  \includegraphics{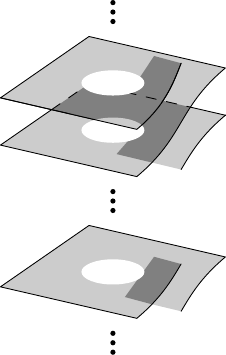}
  \caption{The Riemann surface of the time
  coordinate $Z$. The hole in the middle correspond
  to the image of $\Cp\setminus D(0,r)$ in the $z$-coordinate, while
  a neighbourhood of $z=0$ is sent to a neighbourhood
  of infinity.
  A curve going $k$ times around
  the hole in the $Z$-coordinate will turn
  one time around $\infty$ in the $z$-coordinate.}
  \label{fig:stationnement}
\end{figure}

Now we define the sectors in the $z$-space (see Figure~\ref{fig:sec S_0}).
On the Riemann surface of $Z_j$,
we write $G_j$ for the expression of $g:= f\circ f$ 
in the $Z_j$-coordinate. Let $z_j^\ast = \delta e^{ij\pi\over k}$
for $-k\leq j \leq k$ and some small enough $\delta>0$. Let $Z_j^\ast$ be
the image of $Z_j(z_j^\ast)$. We consider a vertical 
line $\ell_j$ passing through $Z_j^\ast$ and its image $G_j(\ell_j)$. Let $B_j$ be
the domain bounded by $\ell_j$ and $G_j(\ell_j)$ and
containing $\ell_j$ and $G_j(\ell_j)$.
The sector in the $Z_j$-coordinate is then obtained by
$$
  U_j = \left\{ Z_j\ |\ \exists n\in \Z,\, 
          G_j^{\circ n}(Z_j) \in B_j\right\}
$$
for $-k\leq j \leq k$ (see Figure~\ref{fig:U_j}). We see that $U_{-k} = U_k$,
since $Z_k = Z_{-k}$.
The sector $S_j$ in the $z$-coordinate is $Z_j\inv(U_j)$
(see Figure~\ref{fig:sec S_0}). 
These sectors are sometimes called \emph{Fatou petals}.
They are described in great details in \cite{iteration},
although the authors only consider \emph{attractive} petals. Note
that there are $2k$ petals, with half of them being
\emph{repulsive} (see Figure~\ref{fig:petal}).
Also, $S_k$ and $S_{-k}$ are the same petal.

\begin{figure}[htbp]
  \centering
  \includegraphics{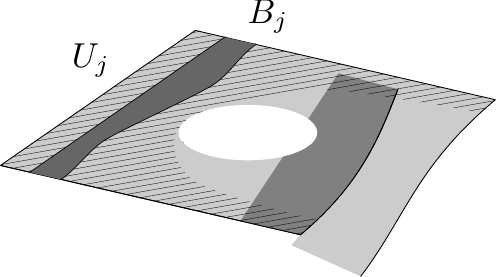}
  \caption{A chart $U_j$ on the Riemann surface with
  the vertical strip $B_j$}
  \label{fig:U_j}
\end{figure}

\begin{rem}\label{rem:Zj}
  Note that
  \begin{align*}
    2i\pi b\, & 
      = \int_{\partial D(0,\delta)} {1 + bz^k \over z^{k+1}}{\rm d} z
        = \sum_{j}\int_{Z_j(\gamma_j)}{\rm d} Z_j 
      = \sum_{j=-k+1}^{k} \Bigl(Z_{j}(z_{j+1}) - Z_{j}(z_{j})\Bigr)
  \end{align*}
  where $\gamma_j$ is an arc of the circle
  $\del D(0,\delta)$ in $S_j$, with endpoints $z_{j+1}$ and $z_{j}$,
  where $z_j = \delta e^{i(2j-1)\pi\over 2k}$. The $Z_j$ defined 
  as in~\eqref{eq:Zj} satisfy this condition.
\end{rem}

\begin{figure}[ptbh]
  \includegraphics{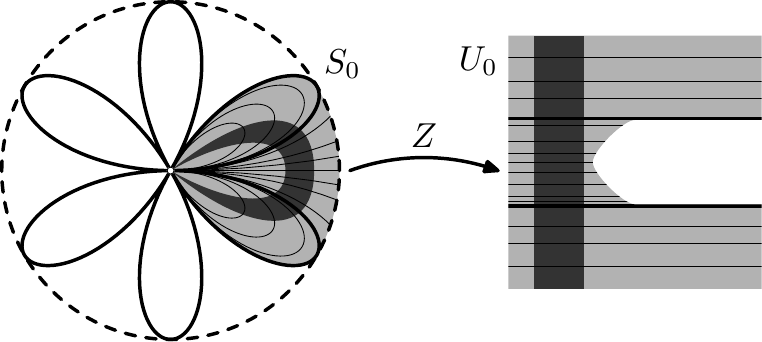}
  \caption{The particular case of $\zpoint = z^4$.
  On the right, the sector $U_0$ in
  the $Z_0$-coordinate, obtained from a strip (in dark gray). 
  On the left, $S_0 = Z\inv(U_0)$ the sector in $z$,
  with the preimage of the strip (in dark gray).}
  \label{fig:sec S_0}
\end{figure}

\begin{figure}[ptbh]
  \includegraphics{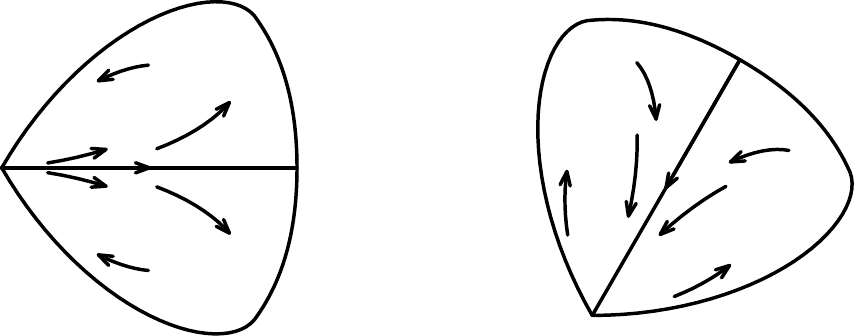}
  \caption{Petals for the holomorphic map $f\circ f$. 
  Dynamics inside a repulsive petal on the left.
  Dynamics inside an attracting petal on the right.}
  \label{fig:petal}
\end{figure}

The sectors are ordered as in Figure~\ref{fig:sec ordre}.
Note in particular that $S_0$ intersects the positive
real axis, and $S_k = S_{-k}$, the negative real axis.

\begin{figure}[ptbh]
  \includegraphics{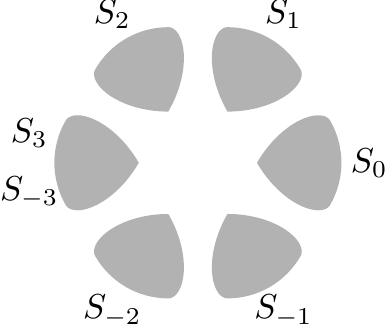}
  \caption{Ordering of the sectors for $k=3$.}
  \label{fig:sec ordre}
\end{figure}

\begin{deff} The \emph{time coordinate} is the Riemann
surface obtained from the disjoint union of the $U_j$,
glued together by the transition functions: the
charts are the $U_j\hookrightarrow \Cp$, with the
diffeomorphism $Z_j\colon S_j \to U_j$ given by
$$
  Z_j(z) = -{1\over kz^k} + b\log z - {ji\pi b\over k},
$$ 
and the transition functions are 
$Z_{j}\circ Z_{j-1}\inv = T_{-i\pi b\over k}$
for $-k < j \leq k$ where the composition is defined.
\end{deff}

The time coordinate is conformally equivalent to
a punctured disk of the origin.

Now we define the complex conjugation on the
time coordinate. Note that on a subdomain
$S_0' \subseteq S_0$ such that $\sigma(S_0') = S_0'$,
we have $Z_0(\zbar) = \overline{Z_0(z)}$. The complex
conjugation on the time coordinate is then obtained
by analytic continuation on the other charts $U_j$.

\begin{prop}[Complex conjugation]
  For $z \in S_j$, let $\ell$ be such that $\sigma(z) = \zbar \in S_\ell$.
  We define the \emph{complex conjugation} $\Sigma$ on 
  the time coordinate in the charts by 
  $$
    \Sigma_{\ell,j}\circ Z_j(z) = Z_{\ell}\circ \sigma(z).
  $$
  Then $\Sigma$ is well-defined and $\Sigma\circ \Sigma = id$.
\end{prop}
\begin{proof}
The proof consists of showing that 
$\Sigma$ is compatible on both charts when $Z_j \in U_j \cap U_{j+1}$ or when
$\Sigma(Z_j) \in U_\ell\cap U_{\ell\pm 1}$. It is a simple computation.
Note that for a subdomain $S_j'\subset S_j$ such that
$\sigma(S_j')\subset S_{-j}$, then in the charts,
we have $\Sigma_{-j,j}(Z_j) = \overline{Z_j}$.
\end{proof}

This allows us to talk about the normal form
$\sigma\circ \vt$ in the time coordinate. It
is the antiholomorphic map $\STt$.

\subsection{Fatou Coordinates} 
Let us call the petals of the normal form $S_j^v$.
The orbits of the normal form $\sigma\circ \vt$
jump from $S_j^v$ to $S_{-j}^v$.
This means that the dynamics of those two
petals is no longer independent, unlike the
holomorphic case. See Figure~\ref{fig:orbit f}.

In its prenormalized form $f$ is close to
$\sigma\circ\vt$ in the sense that 
$|f - \sigma\circ \vt| = o(|z|^{2k+1})$.
In the following lemma, we prove that 
it is also true that $F$ and $\STt$
are close in the time coordinate.

\begin{figure}[htb]
  \includegraphics{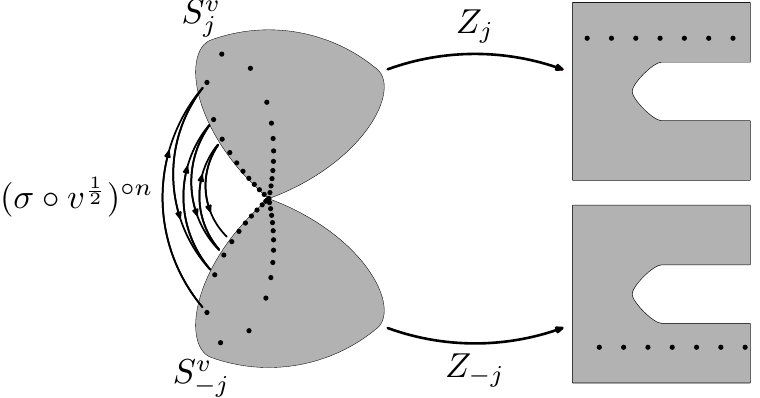}
  \caption{On the left, an orbit of $\sigma\circ \vt$ in the $z$ coordinate.
    The orbit jumps between the sectors $S_j^v$ and $S_{-j}^v$
    of $\sigma\circ \vt$.
    On the right, the same orbit is represented in the 
    time coordinate; it is the orbit of $\Sigma \circ \Tt$.}
  \label{fig:orbit f}
\end{figure}

\begin{lem}\label{lem:F-STt}
  Let $f$ be in its prenormalized form~\eqref{eq:prenormal} and
  let $F$ (resp.\null{} $\STt$) be the expressions of
  $f$ (resp.{} $\sigma\circ \vt$) in the time coordinate.
  On each chart $U_j$, we have
  $|F - \STt| = O(|Z|^{-1})$.
\end{lem}
\begin{proof} The proof is similar to that found in \cite{LectDiffEq}.
  Let $m(z) = f\circ (\sigma\circ \vt)\inv(z) = z + o(z^{2k+1})$.
  We see that 
  \begin{align*}
    Z_j\circ m(z) &= -{1\over kz^k}(1 + o(z^{2k})) 
        + b\log z + o(z^{2k}) - {ji\pi b\over k}\\[2\jot]
      &= Z_j(z) + o(z^k).
  \end{align*}
  Since $z^kZ_j(z) \to -{1\over k}$ when $z\to 0$, and
  because $Z_j$ is invertible and $|Z_j(z)|\to \infty$ 
  when $z\to 0$, it follows
  that $o(z^k)$ is $O(|Z_j|\inv)$ when $|Z_j|\to \infty$. Therefore, $F\circ (\STt)\inv(Z_j)
  = Z_j + O(|Z_j|\inv)$.
\end{proof}

We now present the existence of the Fatou
coordinates. Note that Hubbard and
Schleicher proved their existence in \cite{multiNotPath} 
(Lemma 2.3) in the codimension~1 case for a map with a parabolic
periodic orbit of odd period $n$. We recover their case by
considering $f^{\circ n}$.
This corresponds for us to a germ of
antiholomorphic parabolic diffeomorphism of codimension~1.
The proof in higher codimension is in the same 
spirit with an adaptation, since we need to work
with pairs of sectors $(U_j,U_{-j})$.

\begin{prop}\label{prop:coord Fatou}
  Let $F$ and $\Sigma$ be the expression of $f$ 
  and $\sigma$ in the time coordinates respectively.
  Recall that $U_j = Z_j(S_j)$. On each  $U_j$,
  there exists a holomorphic diffeomorphism
  $\Phi_{j}\colon U_{j} \to \Cp$ such that 
  \begin{equation}\label{eq:F Phi}
    \Phi_j \circ F \circ \Phi_{-j}\inv = \STt,
  \end{equation}
  whenever the composition is defined.

  Moreover, if $\widetilde \Phi_{j}$
  are other Fatou coordinates, then there exists
  $C_j\in \Cp$ for $j=1,\ldots,k-1$ and $C_0,C_k\in\R$ such that 
  $\Phi_j \circ {\widetilde \Phi_j}\inv = T_{C_j}$
  and $\Phi_{-j} \circ {\widetilde \Phi_{-j}}\inv = T_{\overline C_j}$
  for $j\geq 0$.
\end{prop}
\begin{proof}
  The proof makes use of the rigidity of the 
  conformal structure of the doubly punctured sphere
  $S^2\setminus\{0,\infty\}$, as in the
  proof of the uniqueness in the holomorphic case.

  Let $\Phi_j$ denote the Fatou coordinate
  of $f \circ f$ on $U_j$, that is
  $$
    \Phi_{j} \circ (F \circ F) \circ {(\Phi_{j})}\inv = T_1.
  $$
  (We know that it exists since $f \circ f$ is holomorphic and
  that it is unique up to left-composition with a translation.) 
  In the space of the Fatou coordinate $W_{j} = \Phi_{j}(Z_j)$,
  an orbit $\{ (f\circ f)^{\circ n}(z) \}$ corresponds to $\{ W_{j} + n \}$.
  Note also that $\Phi_j(Z_j) = Z_j + D_j + O(|Z_j|\inv)$,
  for some constant $D_j\in\Cp$ (see \cite{LectDiffEq}).

  We first note that each $\Phi_j(U_{j})$
  contains a vertical strip $B_j$ of width~1, by construction
  of the time coordinate $U_{j}$ and of the Fatou coordinate.
  We define 
  $$
    Q_{j} = \Phi_{- j} \circ F\circ {(\Phi_{j})}\inv
    \quad\rlap{\quad for $-k\leq j\leq k$.}
  $$
  Then we see that $Q_{j} \circ Q_{-j} = T_1$ since 
  $\Phi_{j}$ are Fatou coordinates of $F \circ F$.
  It follows that $Q_{j}$ commutes with $T_1$,
  since
  \begin{align*}
    T_1 \circ Q_{j}
      &= (Q_{j} \circ Q_{-j}) \circ Q_{j} \\
      &= Q_{j} \circ (Q_{-j} \circ Q_{j}) \\
      &= Q_{j} \circ T_1.
  \end{align*}
  Indeed, $Q_{j}$ represents $F$ in the Fatou coordinates.
  It is therefore natural that $Q_{j}$ commutes with
  $T_1$, which represents $F \circ F$ in the Fatou coordinates.

  Because $Q_j$ is the composition of an antiholomorphic germ
  by a holomorphic diffeomorphism,  $Q_j$ is antiholomorphic.
  In particular, $\Sigma\circ Q_j$ is holomorphic, and
  $\Sigma\circ Q_{j} - id$ is 
  $1$-periodic and holomorphic, so it has
  a Fourier expansion
  $$
    \Sigma\circ Q_{j}(W_{j}) - W_{j}
      = \sum_{n=-\infty}^\infty c_{n,j} e^{2i\pi n W_{j}}.
  $$
  Moreover, by lemma~\ref{lem:F-STt}, we have
  $\Sigma\circ Q_j(W) = W + M_j + O(|W|\inv)$, where
  $M_j\in\Cp$ is a constant.
  Therefore, $|\Sigma\circ Q_j - id|$ is bounded when
  $|W| \to \infty$, so we must have $c_{n,j} = 0$ 
  for $n\in\Z^\ast$.

  We conclude that $Q_{j}(W_j) = \overline W_j + \overline{c_{j,0}}$.
  Since $Q_j\circ Q_j = T_1$, it follows that $c_{j,0} = {1\over 2} + iy$.
  We then adjust all the $\Im c_{j,0}$ to 0 by choosing the appropriate 
  Fatou coordinates (i.e.\null{} composing them with a
  translation).

  The uniqueness comes from a combination of the uniqueness
  of the Fatou coordinates for the holomorphic $f \circ f$ 
  and having to preserve the constants $c_{j,0} = {1 \over 2}$.
\end{proof}

\section{Modulus of Analytic Classification}\label{sec:module}
If two antiholomorphic parabolic germs are analytically 
conjugate, then they have the same space of orbits. The 
space of orbits of an antiholomorphic parabolic germ
$f$ is a quotient of the set of orbits of the associated 
holomorphic parabolic germ $g = f \circ f$. Hence we start by describing 
the space of orbits of $g$; on a Fatou coordinate, it
is the quotient by $T_1$, which is a bi-infinite 
cylinder. We also need to identify some orbits represented 
in two different Fatou coordinates. This is done by means of 
the transition maps (the horn maps of Écalle).

We will describe the space of orbits of $f$ in 
Section~\ref{sec:espace des orbites} and classify the 
antiholomorphic germs in Section~\ref{sec:class theo}.
To do both of these, we will need the \emph{transition functions},
which will form an analytic invariant.

The transition functions we describe here are the
same as for the holomorphic case. We will introduce 
what we need here; all the details are found in \cite{LectDiffEq}
or \cite{nonlinear}.

In the time coordinate, if $U_j$ is a repelling (resp.\null{}
attractive) petal,
then $U_j$ and $U_{j+1}$ intersect on a domain
containing an upper half-plane (resp.\null{} a
lower half-plane), see Figure~\ref{fig:intersection}.
We can compare the Fatou coordinates $\Phi_j$ and $\Phi_{j+1}$
by looking at
$$\def\to^#1{\buildrel#1\over\longrightarrow}
\def\to^#1{\buildrel#1\over\longrightarrow}
  \displaylines{
  \Psi_j\colon V_j \to^{\Phi_j\inv} U_j \cap U_{j+1} \to^{\Phi_{j+1}} V_{j+1},
\cr\noalign{\penalty1000\noindent resp.\null{}}
  \Psi_j\colon V_{j+1} \to^{\Phi_{j+1}\inv} U_{j+1} \cap U_j \to^{\Phi_j} V_j,
  }
$$
where $V_j = \Phi_j(U_j)$ for all $j$.
This yields a diffeomorphism defined on a domain of
$V_{j}$ (resp.\null{} $V_{j+1}$) containing an upper-half 
plane (resp.\null{} lower half-plane) with its image
in $V_{j+1}$ (resp.\null{} $V_j$) also containing 
some upper-half plane (resp.\null{} lower-half plane).

\begin{figure}[htbp]
  \centering
  \includegraphics{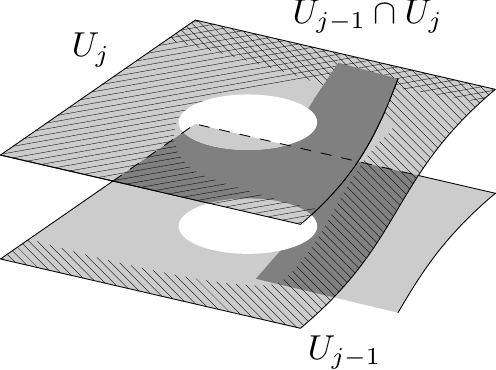}
  \caption{Charts $U_{j-1}$ and $U_j$ on the
  time coordinate. They intersect on a 
  region containing (in this case) an
  upper half-plane.}
  \label{fig:intersection}
\end{figure}

Notice the order of the composition for $\Psi_j$:
we choose the convention that
these functions will go from a repulsive petal
to an attractive petal. Figure~\ref{fig:ordre psi}
is an illustration of the direction of the arrows
in the $z$-coordinate for $k=3$, where $\xi_j$ is
the expression of $\Psi_j$ in the $z$-coordinate.

\begin{deff}\label{def:Psi}
  Let $\Phi_j$ be a Fatou coordinate of $f$ on $U_j$. 
  The \emph{transition functions} (equivalent to the 
  Écalle horn maps) of $f$ are the $2k$ functions $\Psi_j$ 
  for $j=1,\ldots,k$ and $j=-1,\ldots,-k$ obtained by
  \begin{equation}
    \Psi_j = \begin{cases}
      \Phi_{j} \circ \Phi_{j-\sgn(j)}\inv, & \hbox{for $j$ odd};\cr\noalign{\vskip2\jot}
      \Phi_{j-\sgn(j)} \circ \Phi_{j}\inv,&\hbox{for $j$ even};
    \end{cases}
  \end{equation}
  where the composition is defined.
  Here, $\sgn(j)$ is the sign of $j$. 
\end{deff}

\begin{figure}[tbhp]
  \centering
  \includegraphics{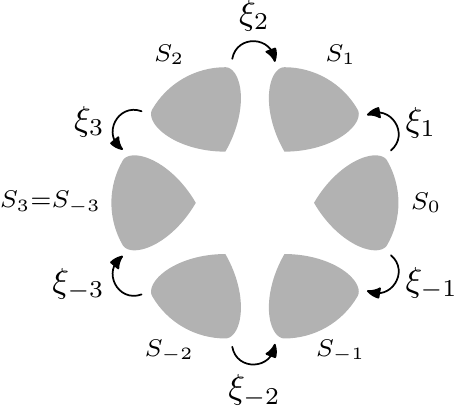}
  \caption{Direction of the transition functions $\{\xi_j\}_j$
  represented in the $z$-coordinates.}
  \label{fig:ordre psi}
\end{figure}

By the uniqueness of Proposition~\ref{prop:coord Fatou}, 
we may change $\Phi_{\pm j}$ by $T_{C_j} \circ \Phi_j$ and 
$T_{\overline C_j} \circ \Phi_{-j}$ for some $C_j\in \Cp$
for $j=1,\ldots,k-1$, or $\Phi_{j}$ by $T_{R_j} \circ \Phi_j$ 
for some $R_j\in \R$ for $j=0,k$.
This will yield another 
set of $2k$ transition functions. We will identify
together these different possible choices of transition
functions at the end of this section.

The following proposition is the first step 
towards the geometric invariant. The transition
functions allow to describe the space of orbits
of $F$ and $F\circ F$.

\begin{prop}\label{prop:Psi commute}
  Let $(\Psi_{-k},\ldots,\Psi_{-1},\Psi_1,\ldots,\Psi_k)$
  be transition functions of $f$. They satisfy
  the equation
  \begin{equation}\label{eq:Psi commute}
    \Sigma \circ \Tt \circ \Psi_j
    = \Psi_{-j} \circ \Sigma \circ \Tt.
  \end{equation}
  In particular, they are transition functions
  of $f\circ f$ and satisfy
  \begin{equation}\label{eq:Psi commute T1}
    T_1 \circ \Psi_j = \Psi_j \circ T_1.
  \end{equation}
\end{prop}
\begin{proof}
  The proof is identical to the holomorphic case; it follows from
  the definition of $\Psi_j$ and~\eqref{eq:F Phi}.
\end{proof}

Equation~\eqref{eq:Psi commute} says that the orbits
of $\Sigma \circ \Tt$ in one Fatou coordinate
are sent by the $\Psi_j$ on the orbits of $\Sigma\circ \Tt$ in
another Fatou coordinate. In those coordinates, the orbits of
$\Sigma \circ \Tt$ correspond to those of $f$. Therefore,
the transition functions allow to identify the
same orbits of $f$ in different coordinates.

We can rewrite Equation~\eqref{eq:Psi commute} as
\begin{equation}\label{eq:Psi moins}
  \Psi_{-j} = \Sigma\circ \Tt \circ \Psi_j \circ \Sigma \circ \mTt,
\end{equation}
so that $\Psi_{-j}$ is determined by $\Psi_j$.
Thus we only need half of the transition functions of $f$ to
determine all of them. For the rest of the
paper, we will work with $\Psi_1, \ldots,\Psi_k$,
knowing that $\Psi_{-1},\ldots,\Psi_{-k}$ are obtained
from Equation~\eqref{eq:Psi moins}.

The transition functions in the holomorphic case
are well known and their properties are described in
\cite{nonlinear} by Ilyashenko. Because the transition
functions of $f$ are also those of $f\circ f$, 
they share the properties which we describe now.

Each $\Psi_j$ satisfies Equation~\eqref{eq:Psi commute T1};
it follows that $\Psi_j - id$ is 1-periodic and has a
Fourier expansion
\begin{equation}\label{eq:Psi fourier}
  \begin{aligned}
    \Psi_j(W) - W &= c_j + \sum_{n=1}^\infty c_{n,j} e^{2i\pi nW}
      \quad\ \rlap{for $j > 0$ odd;}\\
    \Psi_j(W) - W &= c_j + \sum_{n=-1}^{-\infty} c_{n,j} e^{2i\pi nW}
      \quad\ \rlap{for $j > 0$ even.} 
  \end{aligned}
\end{equation}
In particular, we see that $|\Psi_j -id - c_j|$ is
exponentially decreasing when $\Im W \to \infty$
and $j$ is odd (resp.\null{} $\Im W \to -\infty$ and $j$ is even).

Since the Fatou coordinates are not unique, we may change them and
obtain new transition functions. This will change the constants
$c_j$ and $c_{n,j}$, but the following alternating sum will always be preserved
\begin{equation}\label{eq:b}
  \begin{aligned}
    (-1)^{k-1} c_{-k} &+ (-1)^{k-2} c_{-k+1} \pm \cdots - c_{-2} + c_{-1}\\
      &{} - c_1 + c_2 - \cdots + (-1)^{k-1} c_{k-1} + (-1)^{k} c_{k}
        = 2i\pi b.
  \end{aligned}
\end{equation}

We successively change $\Phi_{1}$
by $T_{-c_1 - {i\pi b\over k}}\circ \Phi_1$,
$\Phi_2$ by $T_{-c_1+c_2-{2i\pi b\over k}}$,
$\Phi_3$ by $T_{-c_1+c_2-c_3-{3i\pi b\over k}}\circ \Phi_3$,
and so on. The constant terms of the new
transition functions will satisfy,
\begin{equation}\label{eq:cj}
  c_j = \begin{cases}
    \displaystyle-{i\pi b \over k},
      & \hbox{for $j>0$ odd;}\\[3\jot]
    \displaystyle\phantom{-}{i\pi b \over k},
      & \hbox{for $j>0$ even.}
  \end{cases}
\end{equation}
for $j=1,2,\ldots,k$.
Note that, together with Equation~\eqref{eq:Psi moins}, 
the constant terms $c_j$ for $j=-k,\ldots,-1$
are given by $c_j = \overline{c_{-j}}$.

\begin{deff}[Normalized transition functions]\label{def:normal}
  The transition functions $(\Psi_1,\ldots,\Psi_k)$ are
  said to be \emph{normalized} if the constant terms 
  $c_j$ of Equation~\eqref{eq:Psi fourier} 
  satisfy~\eqref{eq:cj} for $j=1,2,\ldots,k$.
\end{deff}

Even when normalized, the transition functions are not
uniquely determined. There is still a remaining degree of freedom:
we may change the source and target space of each $\Psi_j$
by the same translation $T_C$, with $C \in\R$.

In the case of even codimension, if $f$ is prenormalized, then
$f_1(z)=-f(-z)$ is also prenormalized, and $f$ and $f_1$
are conjugate. Hence, we will need to identify their moduli.

\begin{deff}[Modulus of classification]\label{def:module}
  We define the equivalence relation $\sim$ on normalized
  transition functions of the form~\eqref{eq:Psi fourier} by
  \begin{equation}\label{eq:relation}
    (\Psi_1,\ldots,\Psi_k)\sim(\Psi_1',\ldots,\Psi_k')
    \Leftrightarrow \exists C\in \R, \Psi_j = T_C \circ \Psi_j' \circ T_{-C}.
  \end{equation}
  The \emph{modulus of classification} of $f$ 
  of codimension $k$ (and of positive type when $k$ is even)
  (see Def.\null{}~\ref{def:type +}) is defined by 
  the triple $(k,b,[\Psi_1,\ldots,\Psi_k])$, where
  $k$ is the codimension (Def.\null{}~\ref{def:codim}), 
  $b$ is the formal invariant (Def.\null{}~\ref{def:formel}), and
  $[\Psi_1,\ldots,\Psi_k]$ is the equivalence class
  of normalized transition functions of $f$ 
  (Defs.\null{}~\ref{def:Psi} and~\ref{def:normal}),
  where the equivalence class is defined according to
  the following two cases:
  \begin{enumerate}
    \item If $k$ is odd, then $[\Psi_1,\ldots,\Psi_k]$ is the
      equivalence class under the Relation~\eqref{eq:relation};
    \item If $k$ is even, then $[\Psi_1,\ldots,\Psi_k]$ is
      the equivalence under the Relation~\eqref{eq:relation}
      and the additional relation
      \begin{equation}\label{eq:rel k pair}
        \hskip\parindent (\Psi_1,\ldots,\Psi_k)\cong (\Psi_1',\ldots,\Psi_k')
	\iff \Psi_j' = \Sigma\circ \Tt \circ \Psi_{k-j+1} \circ \Sigma\circ \mTt.
      \end{equation}
  \end{enumerate}
  This equivalence class is called the \emph{analytic invariant}.
\end{deff}

\subsection{Remarks on the Écalle-Voronin Modulus}\label{rem:module Ecalle}
  In the holomorphic case, the modulus of classification
  is known as the \emph{Écalle-Voronin modulus}
  (see \cite{nonlinear} and \cite{LectDiffEq}). For
  any holomorphic parabolic germ $g$ (not necessarily
  of the form $g = f\circ f$), we can obtain its
  Écalle-Voronin modulus the same way as described
  above, but without equation~\eqref{eq:Psi commute},
  so that the $2k$ transition functions are needed.
  There are $2k$ degrees of freedom, so we normalize
  the transition functions by choosing the constant 
  terms as in~\eqref{eq:cj} and with $c_{-j}=-c_j$; 
  the remaining degree of 
  freedom is a translation in every Fatou coordinate
  by a constant $C\in\Cp$.
  Therefore, we quotient by the equivalence relation
  \begin{equation}\label{eq:relation g}
    \begin{aligned}
      (\Psi_{-k},\ldots,\Psi_{-1},\Psi_1,\ldots,\Psi_k)
        &\sim_\Cp(\Psi_{-k}',\ldots,\Psi_{-1}',\Psi_1',\ldots,\Psi_k')\\
      &\ \Leftrightarrow \exists C\in \Cp, \Psi_j = T_C \circ \Psi_j' \circ T_{-C}\\
      &\quad\qquad\hbox{\footnotesize for $j\in\{-k,\ldots,-1\}\cup\{1,\ldots,k\}$}.
    \end{aligned}
  \end{equation}
  We also quotient by the action of the rotations of order $k$.
  If we note the indices $\Psi_{-j}=\Psi_{2k-j+1}$ for $j=1,\ldots, k$,
  then we have the identification 
  \begin{equation}\label{eq:relation g rot}
    \begin{aligned}
      (\Psi_{k+1},\ldots,\Psi_{2k},&\Psi_1,\ldots,\Psi_k)\\
        &\sim (\Psi_{k+1+2m}',\ldots,\Psi_{2k+2m}',\Psi_{1+2m}',\ldots,\Psi_{k+2m}')\\
        &\qquad\qquad\rlap{\kern7.2pc for $m=0,\ldots,k-1$,}\kern-1pc
    \end{aligned}
  \end{equation}
  where indices are mod $2k$.
  We will note the equivalence class of both of these identifications
  by $[\Psi_{-k},\ldots,\Psi_{-1},\penalty0\Psi_1,\ldots,\Psi_k]$. The 
  modulus of $g$ is then $(k,b,[\Psi_{-k},\ldots,\Psi_{-1},\Psi_1,\ldots,\Psi_k])$.

  We describe the link between the Écalle-Voronin modulus of
  $g$ and the modulus of classification of an antiholomorphic
  parabolic germ $f$ of positive type such that $g = f\circ f$, when such a
  $f$ exists. For $k$ odd (resp.\null{} $k$ even), 
  the symmetry axis of $f$ appears along one of the $k$ (resp.\null{}~%
  $\smash{k\over 2}$) ``symmetry axes'' of $g$ of the form $e^{2i\ell\pi\over k}\R$,
  $\ell = 0,\ldots,k-1$. Therefore, we associate to the analytic invariant of $g$
  the regular $k$-gon with the $k$ symmetry axes
  $e^{i\ell \pi\over k}\R$ (see Figure~\ref{fig:polygone module})
  in the following way.
  Divide the $k$-gon by its $k$ symmetry axes to 
  produce $2k$ sectors in the $k$-gon, then starting 
  with the sector above the horizontal line on the
  right side, we identify this sector with $\Psi_1$ and
  going anti-clockwise, we associate $\Psi_2,\ldots,
  \Psi_k$, $\Psi_{-k},\ldots,\Psi_{-1}$ to the subsequent sectors, as in 
  Figure~\ref{fig:polygone module}. The dihedral group $D_{2k}$
  acts on the sectors of the $k$-gon. The action is defined for
  an element $u\in D_{2k}$ by mapping a sector to its
  image by the linear application represented by $u$ (a rotation
  or a reflection).
  
  \begin{deff}\label{deff:action s}
    Let $u\in D_{2k}$ and let
    $\Delta_j$ be the sector of the $k$-gon associated with $\Psi_j$.
    The action of $u$  on the sectors of the regular $k$-gon defines
    a permutation on $\{-k,\ldots,-1,1,\ldots,k\}$
    also noted $u$ by abuse of notation, defined so
    that $u\inv(\Delta_j) = \Delta_{u(j)}$ 
    (see Figure~\ref{fig:polygone module}).
  \end{deff}

  \begin{figure}[htbp]
    \centering
    \subfigure[{The permutation {$u$} for this reflection is
      {$\Psi_{u(1)} = \Psi_{-1}$}, {$\Psi_{u(2)} = \Psi_{-2}$},
      and so on.}]{
      \centering
      \includegraphics{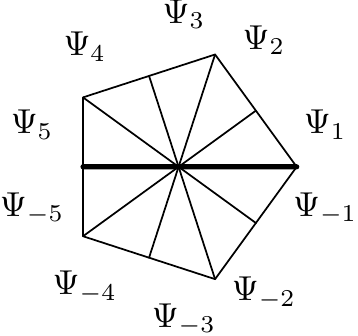}
    }
    \hfill
    \subfigure[{The permutation $u$ for this reflection is
      $\Psi_{u(1)} = \Psi_{2}$, $\Psi_{u(3)} = \Psi_{-1}$,
      and so on.}]{
      \centering
      \includegraphics{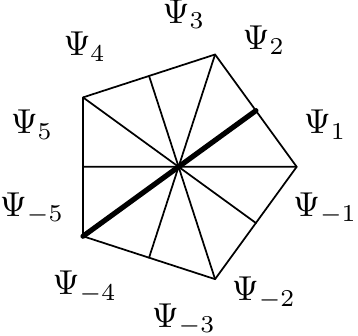}
    }
    \hfill
    \subfigure[{The permutation {$u$} for the rotation of angle
      ${i\pi\over 5}$ is
      {$\Psi_{u(1)} = \Psi_{-1}$}, {$\Psi_{u(2)} = \Psi_{1}$},
      {$\Psi_{u(3)} = \Psi_{2}$},
      and so on.}]{
      \centering
      \includegraphics{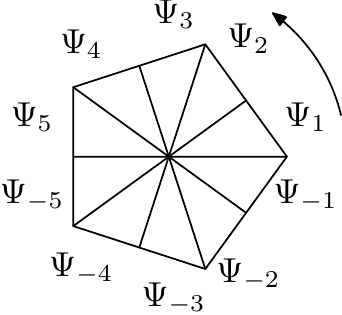}
    }
    \caption{Pentagon with 10 sectors and
    the transition functions, for codimension 5, 
    with the symmetry axis of $f$ in bold in
    (a) and (b).}
    \label{fig:polygone module}
  \end{figure}
  
  Let $(\Psi_{-k},\ldots,\Psi_{-1}$, $\Psi_1,\ldots,\Psi_k)$ be a
  representative of the analytic invariant of $g$. Let us suppose that the symmetry 
  axis of $f$ is along the reflection axis of $s\in D_{2k}$.
  We rotate the coordinate by, say, an angle of ${2\ell i\pi \over k}$,
  so that the symmetry axis corresponds to the (repulsive) real axis 
  (this is possible because $f$ is of positive type). 
  The rotation corresponds to an element $r\in D_{2k}$ and a
  permutation of indices, thus
  we obtain $r\inv sr = s_0$, where $s_0(j)= -j$ is the permutation for
  $s$ in the new coordinate.
  The representative of the modulus of $g$ is permuted: 
  $(\Psi_{r(-k)},\ldots,\Psi_{r(-1)},\Psi_{r(1)},\ldots,\Psi_{r(k)})$
  as in \eqref{eq:relation g rot};
  it now has an equivalent representative 
  $(\Psi_{r(-k)}',\ldots,\Psi_{r(-1)}',\Psi_{r(1)}',\ldots,\Psi_{r(k)}')$
  under Relation~\eqref{eq:relation g} obtained from the Fatou coordinates 
  of $f$, so that it satisfies
  $$
    \Psi_{r(j)}'\circ \STt = \STt \circ \Psi_{r(-j)}'.
  $$
  or equivalently
  \begin{equation}\label{eq:Psi commute s}
    \Psi_j'\circ \STt = \STt \circ \Psi_{s(j)}'.
  \end{equation}
  The analytic invariant of $f$ is then 
  $[\Psi_{r(1)}',\ldots,\Psi_{r(k)}']$.
  \medskip

\fussy

Lastly, we will talk about the modulus of the inverse
of a holomorphic parabolic germ $g$. The following 
proposition is probably well-known, but we could not find it
in the literature, so the proof is included.

\begin{prop}\label{prop:ginv}
  Let $g$ be a holomorphic parabolic germ.
  Let the modulus of $g$ be 
  $(k,b,[\Psi_{-k},\ldots,\Psi_{-1},\Psi_1,\ldots,\Psi_k])$.
  Then the modulus of $g\inv$
  is given by
  \begin{equation}\let\wt=\widetilde
    (k,-b, [\wt\Psi_{-k},\ldots,\wt\Psi_{-1},
      \wt\Psi_1, \ldots, \wt\Psi_{k}]),
  \end{equation}
  where $\widetilde \Psi_j$ stands for $L_{-1}\circ \Psi_{r_1\inv(j)}\inv\circ L_{-1}$,
  with $L_{-1}\colon Y\mapsto -Y$ and with $r_1$, the rotation of the indices
  induced by $y\mapsto \smash{e^{{i\pi \over k}}}y$ 
  (see Definition~\ref{deff:action s}).
\end{prop}
\begin{proof}
\def\gtilde{\widetilde g}
Suppose $g$ is prenormalized.  
The dynamics of $g\inv$ is reversed, therefore the dynamics of $g\inv$
in the first sector $S_0$ is attractive, but the dynamics in
$S_{-1}$ is repulsive. We apply the change of coordinate
$y = L_\lambda(z) = \lambda z$, where $\lambda = e^{i\pi\over k}$. 
Let $\gtilde\inv = L_\lambda \circ g\inv\circ L_\lambda\inv$.

At the formal level, we apply $(z,t)\mapsto (y,-t)$ to  the vector 
field~\eqref{eq:champ} to obtain
\begin{equation}\label{eq:champ inv}
  \dot y = w(y) = {y^{k+1}\over 1 - by^k}.
\end{equation}
The formal normal form of $\gtilde\inv$ is the time-1 map $w^1$
of~\eqref{eq:champ inv}. In particular, $\gtilde\inv$ has
formal invariant $-b$. We will denote the sectors
of $\gtilde\inv$ in the $y$-coordinate by $\widetilde S_j$. 
Note that $\widetilde S_j = L_\lambda\inv(S_j) = S_{j-1}$.

The time coordinate on the sector $\widetilde S_j$ is defined as
\begin{equation}
  Y_j(y) = {-1\over ky^k} - b\log y + {ji\pi b\over k}.
\end{equation}
We have the relation $L_{-1}\circ Y_j = Z_{j-1}\circ L_\lambda\inv$,
where the $Z_j$'s are the time coordinates of $g$. Indeed,
we see that
$$
  Z_{j-1}(e^{-i\pi\over k}y) = {1\over k y^k}
    + b \log y - {i\pi b\over k} - {(j-1)i\pi b\over k}
    = -Y_j(y).
$$

Let $\Phi_j$ be a Fatou coordinate of $g$ in $S_j$. We know
that $\Phi_j\circ Z_j\circ g\inv\circ Z_j\inv \circ \Phi_j\inv = T_{-1}$. 
We will
show that $L_{-1}\circ \Phi_j\circ L_{-1}$ is a Fatou
coordinate of $\gtilde\inv$ on $\widetilde S_{j+1}$. Indeed, we have
\begin{align*}
  \big(L_{-1}\circ \Phi_j\circ L_{-1}\big)\circ Y_{j+1}
  \circ \gtilde\inv\circ Y_{j+1}\inv 
  & {}\circ \big(L_{-1}\circ \Phi_j\circ L_{-1}\big)\inv\\[2\jot]
  &= L_{-1}\circ \Phi_j\circ Z_j\circ g\inv\circ Z_j\inv\circ \Phi_j\inv\circ L_{-1}\\[2\jot]
  &= L_{-1} \circ T_{-1}\circ L_{-1} = T_1.
\end{align*}

It follows that the transition functions $\widetilde \Psi_j$ of $\gtilde\inv$
are given by, for $j\not=1,k$ odd,
\begin{align*}
  \widetilde \Psi_j
    &= \widetilde \Phi_j\circ (\widetilde \Phi_{j - \sgn(j)})\inv\\[2\jot]
    &= L_{-1}\circ \Phi_{j-1} \circ \Phi_{j-1-\sgn(j-1)}\inv \circ L_{-1}\\[2\jot]
    &= L_{-1} \circ \Psi_{j-1}\inv\circ L_{-1}.
\end{align*}
The other values of $j$ are done similarly. Note 
however that for $j=1$, the equation becomes 
$\widetilde\Psi_1 = L_{-1}\circ \Psi_{-1}\inv\circ L_{-1}$,
and that for $j=k$, we have 
$\widetilde \Psi_{-k} = L_{-1}\circ \Psi_k\inv\circ L_{-1}$.
\end{proof}

\color{black}

\bgroup
\let\centering=\raggedcenter 
\section{Space of Orbits and Classification Under Analytic Conjugacy}
\egroup

When two antiholomorphic parabolic germs are analytically conjugate,
it is clear that their space of orbits are essentially the same. Our guiding
principle is that \emph{the space of orbits completely describes the 
dynamics of the germs}; when two germs have the same space of orbits,
they should be analytically conjugate. A formal statement will be given in the
form of the Classification Theorem~\ref{theo:class} in Section~\ref{sec:class theo}.

\subsection{Description of the Space of Orbits} \label{sec:espace des orbites}
The space of orbits in the holomorphic case is
well known. It is briefly described in \cite{nonlinear}.
We will introduce the objects from the holomorphic case
needed for the antiholomorphic description.

In each Fatou coordinate of $f$ (and $f \circ f$), we may choose a
fundamental domain of $g = f\circ f$ by taking any
vertical strip $B_j$ of width~1. We quotient by the 
action of $T_1$, obtaining the bi-infinite Écalle cylinders.
Some orbits of $g$ will appear in two
consecutive cylinders. Since $\Psi_j \circ T_1 = T_1\circ \Psi_j$,
we identify together those orbits by
identifying $W_j$ with $\Psi_j(W_j)$
(or $\Psi_j\inv(W_j)$ depending on $j$).

The universal covering $E\colon \Cp\to \Cp^\ast$
given by $w = E(W) = \exp(-2i\pi W)$ is a biholomorphism
of each cylinder onto $\Cp^\ast$. This allows us to see
the Écalle cylinders as Riemann spheres punctured at $0$ and infinity
$\Ss^2_j\setminus \{0,\infty\}$ (see Figure~\ref{fig:bande}).
We will define the horns maps using this universal covering.

\begin{figure}[htbp]
  \includegraphics{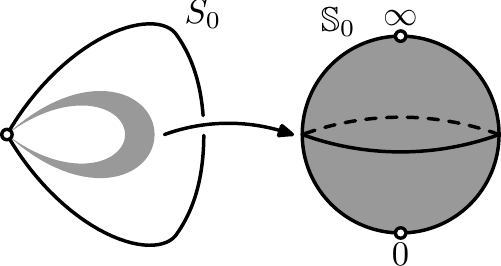}
  \caption{A fundamental domain obtained by $Z_0\inv(B_0)$
  in the $z$-coordinate (in gray)
  and the sphere it represents}
  \label{fig:bande}
\end{figure}

\begin{deff}\label{deff:horn maps}
  The \emph{horns maps} $\psi_j$, for $j=1,\ldots,k$, 
  are defined by
  $\psi_j = E\circ \Psi_j\circ E\inv$,
  where $\Psi_j$ is a transition function and
  \begin{equation}
    E(W) = \exp(-2i\pi W).
  \end{equation}
\end{deff}
For positive $j$ odd (resp.\null{} $j$ even), the horn maps are
defined on a punctured neighbourhood of the origin (resp.\null{} 
of infinity) with their image in a neighbourhood
punctured of the origin (resp.\null{} of infinity). By the
Riemann Removable Singularity Theorem, they extend to
$$
  \psi_j \colon 
  \begin{cases}
      (\Cp,0) \to (\Cp,0), & \hbox{for $j > 0$ odd;}\\[3\jot]
      (\Cp,\infty) \to (\Cp,\infty), & \hbox{for $j > 0$ even.}
  \end{cases}
$$
To retrieve the $\psi_j$ for $j < 0$, we use Equation~\eqref{eq:Psi moins} 
in the coordinate $w = E(W)$
\begin{equation}
  \psi_{-j} = L_{-1}\circ \tau \circ \psi_j \circ L_{-1}\circ \tau,
\end{equation}
where $L_{-1}(w) = -w$ and $\tau(w) = {1 \over \wbar}$.

The space of orbits of $g = f\circ f$ is described by the $2k$ spheres with
identifications at the origin or at infinity, as seen
in Figure~\ref{fig:espace des orbites ff}.

\begin{figure}[htbp]
  \centering
  \includegraphics{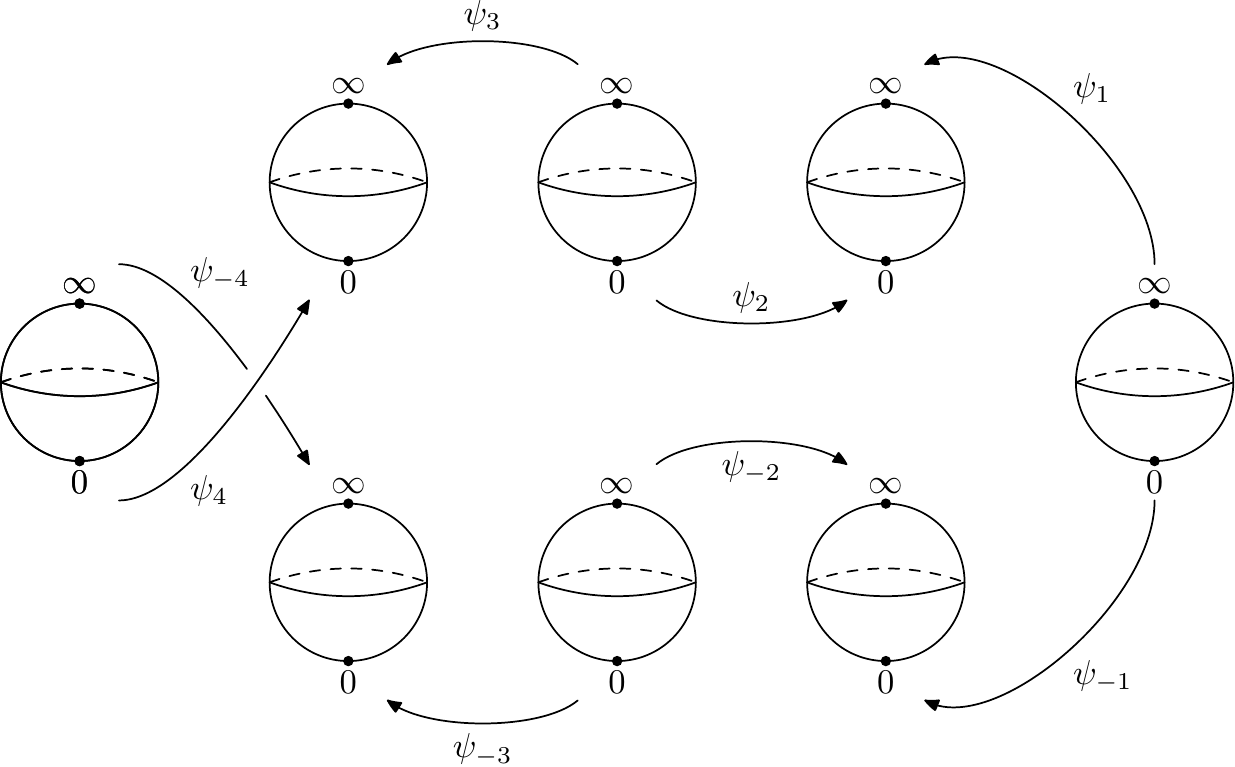}
  \caption{Space of orbits for $f\circ f$ of codimension~4}
  \label{fig:espace des orbites ff}
\end{figure}

\begin{rem}
By differentiating $E\circ \Psi_j \circ E\inv$, we
obtain
\begin{equation}
  \begin{aligned}
  {(\psi_j)}'(\infty) &= e^{{2\pi^2b\over k}}\quad\rlap{ for $j$ odd;}\\
  {(\psi_j)}'(0) &= e^{{2\pi^2b\over k}}\quad\rlap{ for $j$ even;}
  \end{aligned}
\end{equation}
so the product of those derivatives for $j\in\{-k,\ldots,-1\}\cup\{1,\ldots,k\}$
is~$e^{4\pi^2 b}$.
\end{rem}

\subsection{The space of orbits of $f$}\label{sec:espace des orbites f}
The complex conjugation
$\Sigma$ becomes $\tau(w) = {1\over \overline w}$ on the spheres,
where $w\in \Ss_j^2$ and $\tau(w)\in \Ss_{-j}^2$. The translation
$\Tt$ becomes $L_{-1}(w) = -w$. In the Fatou coordinates,
$f$ is $\STt$, so that on the spheres, $f$ is 
$L_{-1}\circ \tau(w) = -{1\over \overline w}$. To obtain the space of
orbits of $f$, we identify $w$ and $L_{-1}\circ \tau(w)$ in the space
of orbits of $f \circ f$, that is on the $2k$ spheres
above.

Let us first consider the case of codimension~1, so that we
have two sectors $S_0$ and $S_1=S_{-1}$ in the $z$-coordinate
and two spheres $\Ss_0^2$ and $\Ss_1^2$.
Recall that for $z\in S_0$, $f\inv(z)$ is still in $S_0$.
This means that $L_{-1}\circ \tau$ acts on $\Ss_0^2$.
It sends $0$ to $\infty$, the northern hemisphere to
the southern hemisphere, and the equator on itself. On
the equator $|w|=1$, we identify $w$ to $-w$. The resulting
surface is the real projective space $\R\Pp^2$.
This is also true for $S_1$ and $\Ss_1^2$. The equators
of both spheres play a special role; they each represent
orbits along an invariant half-curve that each forms a 
``semi-axis'' of reflection for $f$.

Therefore, in codimension~1, the space of orbits is two
real projective spaces with one germ of holomorphic
diffeomorphism 
$$[\psi_1]\colon (\R\Pp^2,[0])\to (\R\Pp^2,[0]),$$
where $[\psi_1]$ is the equivalence class of $\psi_1$ 
under the quotient of $\Ss^2$ to $\R\Pp^2$
and $[0]$ is equivalence class of the points $\{0,\infty\}$ identified
together. The class $[\psi_1]$ is well-defined since
$\psi_1\circ L_{-1}\circ \tau = \tau \circ L_{-1}\circ \psi_{-1}$.

In codimension~$k>1$, the spheres $\Ss_0$ and $\Ss_k$ both
quotient to a real projective space, but the other spheres
are identified in pairs $(\Ss_j,\Ss_{-j})$, so that the
quotient of the union of the two spheres is diffeomorphic 
to a sphere. The space of orbits
is then described by two real projective spaces together with
$k-1$ spheres and $k$ equivalence classes of horn maps
$[\psi_j] = \{\psi_j,\psi_{-j}\}$. The class $[\psi_j]$ 
defines a germ at $[0]=[\infty]$ on the quotiented spheres $(\Ss_j,\Ss_{-j})$, 
since the representatives satisfy
$\psi_j\circ L_{-1}\circ \tau = L_{-1}\circ \tau \circ\psi_{-j}$.

\begin{figure}[htbp]
  \includegraphics{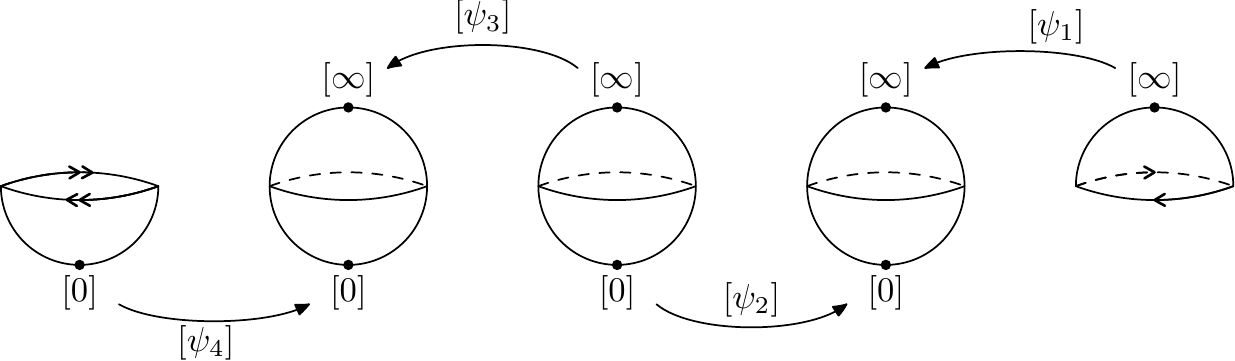}
  \caption{The space of orbits of $f$ in codimension~$4$}
\end{figure}

On the two extreme projective spaces we have a distinguished curve 
given by the equator. The only changes of coordinates on $\mathbb D$
preserving the equator and sending opposite points to opposite points 
are the linear maps $L_c$ with $|c|=1$. Hence the lines $\Im W = y$  
are invariant. This $y$ is the generalization of the \emph{Écalle height} 
introduced in \cite{multiNotPath}.
%
However, on the other spheres, there does not appear to be
a quantity preserved by changes of coordinates.

The two projective spaces correspond to the orbits of two Fatou petals
containing the formal symmetry axis of $f$. The equator
of each projective space corresponds to a semi-axis of symmetry
in each respective petal. These axes meet at the origin, but generally 
they cannot be extended into a real analytic curve. We will
see in Theorem~\ref{theo:f reel} exactly when they extend 
into a real analytic curve. This formal symmetry axis explains 
the existence of the Écalle height in the two projective spaces 
and why the Écalle height does not exist in the other spheres.
\color{black}

\subsection{Classification Under Analytic Conjugacy}
\label{sec:class theo}
We can now state and prove the main theorem of this paper.
\begin{theo2}\label{theo:class}
  For antiholomorphic parabolic 
  germs of codimension~$k$ (and of positive type 
  when $k$ is even), we have:
  \begin{enumerate}
    \item The modulus of classification is a 
    complete invariant of analytic classification
    under holomorphic conjugacy;
    \item The moduli space is the
    set of $(k,b,[\Psi_1,\ldots,\Psi_k])$,
    where the constant terms of $\Psi_j$
    satisfy~\eqref{eq:cj}, and $[\Psi_1,\ldots,\Psi_k]$
    is the equivalence class under the Relation~\eqref{eq:relation}
    (and \eqref{eq:rel k pair} for $k$ even).
  \end{enumerate}
\end{theo2}
\begin{proof}
  The proof is analogous to that in the holomorphic case, 
  which can be found in \cite{nonlinear} and \cite{LectDiffEq}.

  1. Let $f_\ell\colon(\Cp,0) \to (\Cp,0)$ be a
  germ of antiholomorphic diffeomorphism with
  a parabolic fixed point at the origin and let
  $(k_\ell,b_\ell,[\Psi_{1,\ell},\ldots,\Psi_{k,\ell}])$
  be its modulus of classification, for $\ell=1,2$.
  We can of course suppose that $f_\ell$ is prenormalized.

  Suppose first that $f_2(u) = h\circ f_1\circ h\inv(u)$
  for some germ of analytic diffeomorphism $u = h(z)$. 
  The germs $f_1$ and $f_2$ must
  have the same codimension and formal invariant, since they
  are topological and formal invariants.
  For the analytic invariant, first let $F_1$, $F_2$
  and $H$ denote the expressions of $f_1$, $f_2$ and $h$
  in the time coordinate. If $\Phi_j$ is a Fatou
  coordinate of $F_2$ on $U_j$, then $\Phi_j\circ H$ is a
  Fatou coordinate of $F_1$. It follows that they have
  the same transition functions.

  \sloppy 

  Conversely, suppose $f_1$ and $f_2$ have the same
  modulus. We can choose a common normalized representative
  $(\Psi_1,\ldots,\Psi_k)$ for both classes and Fatou coordinates
  $\Phi_{j,\ell}$ for $f_\ell$ ($\ell=1,2$), such that
  the $\Psi_j$ are the transition maps for these
  Fatou coordinates. Let $\varphi_{j,\ell} = Z_j\inv\circ \Phi_{j,\ell}\circ Z_j$, 
  for $\ell=1,2$, be the expression of Fatou coordinates in 
  the $z$-coordinate.  Then we define
  $$
    h_j(z) = \varphi_{j,1}\inv\circ \varphi_{j,2}(z)
  $$ 
  on $S_j$. It sends the orbits of $f_2$ to those of $f_1$
  on $S_j$. On the intersection of two consecutive sectors
  $S_j\cap S_{j+1}$, we see that
  $$
    h_{j+1}\circ h_j\inv(z) = id,
  $$
  so the $h_j$'s agree on the intersection.
  We may define $h$ on $\bigcup_j S_j$
  by $h(z) = h_j(z)$ if $z\in S_j$. It sends a
  punctured neighbourhood of the origin into
  another punctured neighbourhood of the origin,
  so by the Riemann Bounded Extension Theorem, $h$
  extends to a holomorphic diffeomorphism
  of a neighbourhood of the origin.
  
  \fussy

  Finally, we see that $h\circ f_2 = f_1 \circ h$ 
  since it sends the orbits of $f_2$ to the orbits
  of $f_1$ on a whole neighbourhood of the origin.

  2.\null{} The proof is in two steps. First we construct an
  abstract Riemann surface $S$ on which 
  $\STt$ is well defined and we prove that this surface has the
  conformal type of a punctured disk. Then we prove that
  $\STt$ is the germ we are looking for on the disk.

  Consider a triple $(k,b, [\Psi_1,\ldots,\Psi_k])$. We
  must find a parabolic germ of antiholomorphic diffeomorphism
  with this modulus of classification.

  Let $(\Psi_1,\ldots,\Psi_k)$ be a representative of the
  analytic invariant and let $\Psi_{-1},\ldots,\Psi_{-k}$ be
  the other transition functions obtained from~\eqref{eq:Psi moins}. 
  Let $\sigma\circ \vt$ be the normal form of codimension~$k$
  with formal invariant $b$. Let $U_j$ be the charts in the time coordinate
  of $v$. 
  
  We consider the transition functions defined on
  those charts. More precisely, $\Psi_1$ is defined on a domain of $U_0$ 
  containing an upper-half plane with its image in
  $U_1$; $\Psi_2$ is defined on a domain of
  $U_2$ containing a lower half-plane with its image
  in $U_1$, and so on. We define the Riemann surface
  $S$ by
  $$
    S = \bigsqcup_{j=-k}^k U_j\big/\sim
  $$
  where $\sim$ identifies $W_j\in U_j$ with its image by
  $\Psi_j$ or $\Psi_{j-1}$ (depending on $j$).
  As it is done in \cite{LectDiffEq}, we can build
  a smooth quasi-conformal mapping $P\colon S\to \Cp^\ast$,
  and from the Ahlfors-Bers Theorem (see \cite{quasi}), find a diffeomorphism
  $Q\colon D \to (\Cp,0)$, where $D = P(S)\cup\{0\}$,
  so that the composition
  $H = Q\circ P\colon S\to (\Cp,0)$ is holomorphic. In fact,
  $H$ is a biholomorphism of $S$ onto some punctured disk
  of the origin.

  By~\eqref{eq:Psi commute}, we know that $\STt$ is well
  defined on $S$. The map $f = H\circ \STt\circ H\inv$
  extends to the origin by $f(0) = 0$, since it is
  bounded around $0$ (we can apply the Riemann Removable Singularity 
  Theorem to $f\circ \sigma$ to see this).
  Lastly, $T_1$ is also well-defined on $S$, so
  we set $g = H\circ T_1\circ H\inv$. Since 
  $T_1 = (\STt)\circ (\STt)$, it follows
  that $g = f\circ f$, and by the chain rule, 
  $\left|{\del f\over \del\zbar}\right|^2 = g'(0)$.
  By the holomorphic case, we know $g$ is a 
  holomorphic parabolic germ of codimension~$k$, so it
  follows that $f$ is parabolic and of codimension~$k$.
  The formal invariant of $f$ is $b$, since it is
  determined by~\eqref{eq:b}.

  It remains only to prove that $f$ is of positive type.
  The formal symmetry axis of $f$ is the real line,
  since the $\Psi_1$ and $\Psi_{-1}$ are defined on
  $U_0$ and $\Psi_j\circ \STt = \STt\circ \Psi_{-j}$.
  Moreover, the petal $S_0$ of $f$ is
  repulsive, since $H\big|_{U_0}\circ Z_0\colon U_0\to U_0$
  is a Fatou coordinate of $f$ and $U_0$ contains a half-plane
  $\{\Re Z < R\}$, i.e.\null{} $S_0$ contains
  all the backward iterates of $f$.
  Therefore, $f$ is of positive type.

  We conclude $f$ is a germ with modulus
  $(k,b,[\Psi_1,\ldots,\Psi_k])$.
\end{proof}

\section{Applications of the Classification Theorem}\label{sec:applications}
Here we will solve Questions~\sanspt{\getrefnumber{Q:premier}} 
to~\sanspt{\getrefnumber{Q:dernier}} of the introduction and
other related questions.

\subsection{Embedding in a Flow or the Complex Conjugate of a Flow}
If a holomorphic germ $g$ is conjugate to the normal form, that is
$g = h\circ \vun\circ h\inv$, then it is embedded in
the family $g_t = h\circ v^t\circ h\inv$.
Similarly, we will say that $f$ is \emph{embeddable}
if it is embedded in the family 
$f_t:=h\circ \sigma\circ v^t\circ h\inv$.
When is an antiholomorphic parabolic germ embeddable?
This corresponds to Question~\sanspt{\getrefnumber{Q:flot}}. 
The answer is read in the modulus of classification.

\begin{theo}[Embedding in a flow]
  An antiholomorphic parabolic germ $f$ is
  analytically conjugate to its normal form
  (i.e.\null{} embeddable)
  if and only if the transition functions of
  $f$ are translations.
\end{theo}
\begin{proof}
  The transition functions of the normal form
  are translations since the time charts $Z_j$
  are Fatou coordinates. Therefore, it follows from the
  Classification Theorem~\ref{theo:class}.
\end{proof}

In Proposition~\ref{prop:forme reelle}, we proved that
$f$ is always formally conjugate to the sum
of a formal germ with real coefficients. We ask the related question.
\begin{Q}\label{Q:reel}
  When is a parabolic antiholomorphic germ analytically
  conjugate to a series with real coefficients?
\end{Q}

Of course, this is the case for any embeddable germ. But 
we will show in Section~\ref{sec:forme reelle} that 
the embeddable germs form a subset of infinite 
codimension in the set of antiholomorphic parabolic germs 
conjugate to a germ with real coefficients.
We will answer this question in Section~\ref{sec:forme reelle}.
We first tackle Question~\sanspt{\getrefnumber{Q:racine}}.

\subsection{Antiholomorphic {\bfmi n}-th Root Problem}\label{sec:racine}
Question~\sanspt{\getrefnumber{Q:racine}} and its restatement 
Question~\ref{Q:racines bis}, as well as Question~\ref{Q:racines antihol}
are concerned with the existence of antiholomorphic
roots and their uniqueness.

\sloppy
\begin{theo}[Antiholomorphic Root Extraction Problem]
  Let $g$ be a holomorphic parabolic germ and let
  $(k,b,[\Psi_{-k},\ldots,\Psi_{-1},\Psi_1,\ldots\Psi_k])$
  be its Écalle-Voronin modulus (see Section~\ref{rem:module Ecalle}),
  then 
  \begin{enumerate}
  \item $g$ has $k$ one-parameter families of formal antiholomorphic 
  $n$-th roots, one for each of the $k$ formal axes of reflection;
  \item $g$ has an antiholomorphic $n$-th root ($n$ even) tangent to the
  formal symmetry axis $e^{i \pi\ell\over k}\R$
  if and only if there exists a representative 
  $(\Psi_{-k},\ldots,\Psi_{-1}$, $\Psi_1,\ldots,\Psi_k)$
  of the equivalence relation~\eqref{eq:relation g} 
  which has a symmetry property with respect to the formal
  symmetry axis $e^{i\pi\ell\over k}\R$, namely
  \begin{equation}\label{eq:racine existe}
    \Psi_j \circ \Sigma\circ T_{{1\over n}}
      = \Sigma \circ T_{{1\over n}} \circ \Psi_{s_\ell(j)},
  \end{equation}
  where $s_\ell$ is the reflection of indices with respect
  to $e^{i\pi\ell\over k}\R$ (see Definition~\ref{deff:action s});
  \item If $g$ is not analytically conjugate to its normal
  form, then 
  \begin{enumerate}
    \item[{\it i)\/}] Each family has at most one convergent root,
      so $g$ has at most $k$ distinct antiholomorphic $n$-th
      roots $f_j$;
    \item[{\it ii)\/}] If $g$ has $m$ distinct antiholomorphic 
      roots $f_{\ell_1}$,$\ldots$, $f_{\ell_m}$ with distinct 
      linear parts $e^{2i\pi\ell_j\over k}\zbar$,
      $j=1,\ldots,m$, then the modulus of $g$ has 
      $\gcd(\ell_m-\ell_1,\ldots,\ell_m-\ell_{m-1},k)$ 
      independent transition functions. 
  \end{enumerate}
  \end{enumerate}
\end{theo}
\sloppy
\color{black}
\fussy
\begin{proof}
  1.\null{} That $g$ has a 1-parameter family of formal antiholomorphic
  $n$-th root on each symmetry axis is a consequence of the fact
  that $g$ is formally conjugate to the normal form $\vun$,
  and Proposition~\ref{prop:racine v}.

  2.\null{} We first prove this for the case $k$ odd.
  Let 
  $$
    {\bf \Psi} = (\Psi_{-k},\ldots,\Psi_{-1},\Psi_1,\ldots,\Psi_k)
  $$
  be a representative of the analytic invariant of $g$.

  First, we suppose that the representative satisfies \eqref{eq:racine existe}.
  To realize an antiholomorphic root tangent to the symmetry
  axis $e^{2i\pi\ell \over k}\R$, we rotate the coordinate by 
  $R_\ell(z) = e^{-{2i\pi \ell\over k}}z$ so that this axis is 
  on the real line and the dynamics of $R_\ell\circ g\circ R_\ell\inv$ 
  on the side of the positive real axis is repulsive. 
  Let $g_\ell = R_\ell\circ g\circ R_\ell\inv$ and let
  $r$ correspond to the permutation of indices defined by $R_\ell$ 
  (see Definition~\ref{deff:action s}). The representative is permuted into 
  $(\Psi_{r(-k)},\ldots,\Psi_{r(-1)},\Psi_{r(1)},\ldots,\Psi_{r(k)})$.
  Also, we have that $r\inv s_\ell r = s_0$, where $s_0(j) = -j$
  and $s_\ell$ is the reflection of indices induced by $\sigma_\ell$.
  Equation~\eqref{eq:racine existe} becomes
  $$
    \Psi_{r(j)}\circ \STt = \STt \circ \Psi_{r(-j)}.
  $$
  Now we may repeat the proof of part two of Theorem~\ref{theo:class}
  to obtain an antiholomorphic germ $f_\ell\colon (\Cp,0)\to (\Cp,0)$
  with analytic invariant $[\Psi_{r(1)},\ldots,\Psi_{r(k)}]$ such
  that $f_\ell$ is $\Sigma\circ T_{{1\over n}}$ in each Fatou
  coordinate of $g_\ell$. Since $(\Sigma\circ T_{{1\over n}})^{\circ n} = T_1$
  in the Fatou coordinates, it follows that $f_\ell^{\circ n} = g_\ell$.

  Conversely, suppose $g$ has an $n$-the root $f_\ell$ tangent to
  the reflection axis $e^{2i\pi \ell\over k}\R$.
  We rotate the coordinate by $R_\ell$. Let $f = R_\ell\circ f_\ell\circ R_\ell\inv$.
  Now in every Fatou coordinate of $g_\ell$,
  $f$ has the form $\Sigma\circ T_{{1\over n}+iy}$, by Proposition~\ref{prop:racine v}. 
  By changing Fatou coordinates, we may obtain $\Sigma\circ T_{{1\over n}}$. 
  Those Fatou coordinates give us a representative 
  $(\Psi_{-k}',\ldots,\Psi_{-1}',\Psi_1',\ldots,\Psi_k')$ that
  is equivalent to $\bf \Psi$ under the
  Relations~\eqref{eq:relation g} and~\eqref{eq:relation g rot} and
  that satisfies
  $$
    \Psi_{j}'\circ \STt = \STt \circ \Psi_{-j}'.
  $$
  It follows that 
  $(\Psi_{r\inv(-k)}',\ldots,\Psi_{r\inv(-1)}',\Psi_{r\inv(1)}',\ldots,\Psi_{r\inv(k)}')$ 
  is a representative equivalent to $\bf \Psi$ under
  Relation~\eqref{eq:relation g} which satisfies \eqref{eq:racine existe},
  using the fact that $(r\inv s_\ell r)(j) = -j$.

  In the case $k$ even, we may have antiholomorphic roots
  of positive and negative type. Those of positive type
  are tangent to an axis $e^{i\pi\ell\over k}\R$ with
  $\ell$ even and are done as before. Those of negative
  type are tangent to an axis $e^{i\pi\ell\over k}\R$ with
  $\ell$ odd. For such a root $f_\ell$, $f_\ell\inv$ will
  be a root of positive type of $g\inv$.
  By Proposition~\ref{prop:ginv}, the modulus of $g\inv$ is
  $$\let\wt\widetilde
    (k,-b, [\wt\Psi_{-k},\ldots,\wt\Psi_{-1},
      \wt\Psi_1, \ldots, \wt\Psi_{k}]),
  $$
  where $\widetilde \Psi_j = L_{-1}\circ \Psi_{r_1\inv(j)}\inv\circ L_{-1}$
  and $r_1$ is the rotation of indices induced by
  $z\mapsto \smash{e^{i\pi\over k}} z$.
  By the previous two paragraphs,  $g\inv$ has an $n$-th
  antiholomorphic root of positive type tangent to $e^{i \pi\ell\over k}\R$ 
  if and only if there exists a representative
  $(\wt\Psi_{-k},\ldots,\wt\Psi_{-1},
      \wt\Psi_1, \ldots, \wt\Psi_{k})$
  such that
  $$
      \widetilde \Psi_j \circ \Sigma\circ T_{{1\over n}}
        = \Sigma \circ T_{{1\over n}} \circ \widetilde \Psi_{s_\ell(j)},
  $$
  By simplifying the $L_{-1}$ and taking the inverse on both sides,
  we obtain~\eqref{eq:racine existe}.

  3.\null{} $i$) Suppose $g$ has two roots $f_1$
  and $f_2$ from the same family. In particular,
  they have the same linear term, so we may suppose
  they are tangent to $\sigma$, modulo conjugating $g$
  by a rotation of order $k$. In the Fatou coordinates, they take the
  form $\Sigma\circ T_{{1\over2} + iy_j}$, for $j=1,2$,
  by Proposition~\ref{prop:racine v}. Since 
  $\Sigma\circ T_{{1\over2} + iy_j}$ satisfies
  \eqref{eq:Psi commute} for $j=1,2$, and by combining
  with~\eqref{eq:Psi fourier}, we see that either
  $y_1 = y_2$ or the $\Psi_j$'s are translations.

  $ii$) Lastly, consider the dihedral group 
  $D_{2k}$ with its action on the regular 
  $k$-gon. Recall that we can divide a regular $k$-gon
  by its $k$ symmetry axes to form
  $2k$ sectors, which we associate to the transition
  functions (see Section~\ref{rem:module Ecalle} and
  Figure~\ref{fig:polygone module}).

  Let $H = \langle s_1,\ldots,s_m\rangle$ be 
  the subgroup of $D_{2k}$ generated by the reflections along the
  symmetry axes of $f_{\ell_1},\ldots, f_{\ell_m}$. 
  Each element of $H$ acts on the modulus by introducing
  relations of the type
  $$
    \STt\circ \Psi_j = \Psi_{s_\ell(j)}\circ \STt,
  $$
  thus reducing the number of independent $\Psi_j$ (see 
  Figure~\ref{fig:dodecagone}). The orbit by $H$ of a sector
  represents the transition functions tied together, therefore
  the number of independent transition functions corresponds
  to $|D_{2k}:H|$. To compute this, we first observe that
  $H$ must itself be a dihedral group, so that $H = D_{2j}$
  for some $j$ (see~\cite{grove}). In fact, $j$ is given by
  $$
    j = {k\over \gcd(\ell_m - \ell_1,\ldots,\ell_m - \ell_{m-1},k)},
  $$
  since any rotation $s_ms_p$ has order $k/\gcd(\ell_m - \ell_p,k)$.
  It follows that $|D_{2k}:H| = \gcd(\ell_m - \ell_1,\ldots, \ell_m - \ell_{m-1},k)$.
  \begin{figure}[htbp]
  \includegraphics{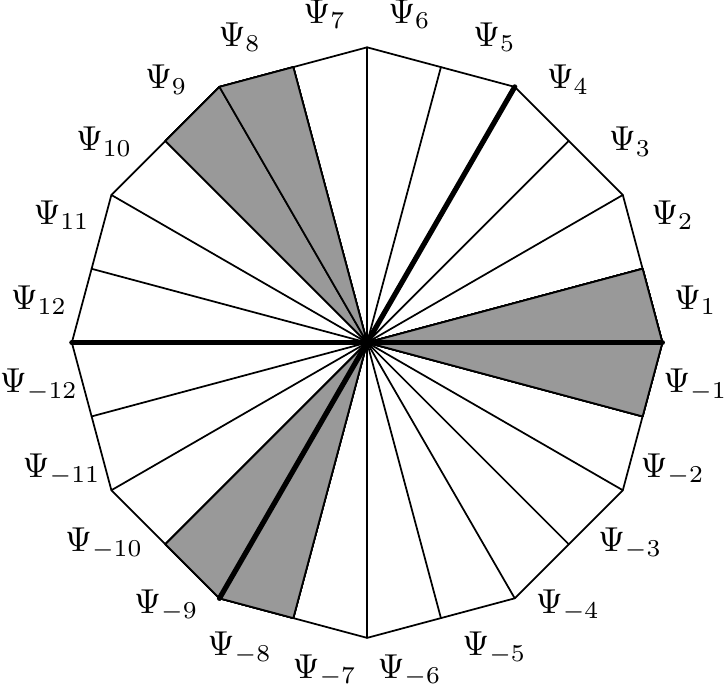}
  \caption{The orbit of one sector by the subgroup
  $H$ generated by two reflections (bold axes)
  in codimension 12.}
  \label{fig:dodecagone}
  \end{figure}
\end{proof}

\begin{theo}\label{theo:racines f}
  Let $f\colon(\Cp,0)\to(\Cp,0)$ be 
  an antiholomorphic parabolic germ of codimension $k$,
  and of positive type when $k$ is even.
  Let $(k, b, [\Psi_1,\ldots,\Psi_k])$
  be the modulus of $f$ and let $n\geq 3$
  be an odd number. Then $f$ has a single
  antiholomorphic formal $n$-th root ($n$ odd). Moreover, $f$
  has an antiholomorphic $n$-root ($n$ odd) 
  if and only if 
  \begin{equation}\label{eq:Psi commute n}
      \Psi_{j} \circ T_{{1\over n}}
      = T_{{1\over n}} \circ \Psi_{j}.
  \end{equation}
\end{theo}
\begin{proof}
  That $f$ has a unique formal $n$-th root follows
  from Proposition~\ref{prop:racine v}.

  For the second part, $f$ has an antiholomorphic $n$-th root if
  and only if the transition functions satisfy 
  $\Psi_{j} \circ \Sigma\circ T_{1\over 2n} = \Sigma\circ T_{1\over 2n}\circ \Psi_{-j}$,
  since we may realize $\Sigma\circ T_{1\over 2n}$ on the Riemann surface
  of $f\circ f$. We combine this last equation with \eqref{eq:Psi commute}
  to obtain 
  $\Psi_{j}\circ T_{{1\over 2n}-{1\over2}} = T_{{1\over 2n}-{1\over2}}\circ\Psi_{j}$.
  To conclude, we note that $\gcd(1-n, 2n)=2$ because $n$ is odd, so
  there exists $p,q\in\Z$ such that $(1-n)p - 2nq = 2$. In other words,
  we have ${1\over n} = {(1-n)\over 2n}p - q$. Since $\Psi_{j}$ commutes with
  $T_{{(1-n)\over 2n}p}$ and with $T_q$ (because $q$ is an integer), it
  follows that $\Psi_{j}$ commutes with $T_{1\over n}$.
\end{proof}

\begin{cor}
  $f$ has an antiholomorphic $n$-th root for $n$ odd if and only if
  $f$ is the square root of $g$ and $g$ has a holomorphic
  $n$-th root, with $n$ odd.
\end{cor}
\begin{proof}
  Equation~\eqref{eq:Psi commute n} is independent of the representative and
  it is equivalent to a holomorphic parabolic germ having a holomorphic $n$-th root 
  (see~\cite{nonlinear}).
  
  Suppose $g = f\circ f$ and $g$ has a holomorphic $n$-th root. Then
  the modulus of $g$ and $f$ satisfies~\eqref{eq:Psi commute n},
  so that $f$ has an antiholomorphic $n$-th root.

  The converse is direct.
\end{proof}
%

\subsection{Germs with an Invariant Real Analytic Curve}\label{sec:forme reelle}
An antiholomorphic germ with real coefficients preserves the real axis. 
Any germ $f$ analytically conjugate to the latter
will preserve a real analytic curve; it is a
property of the equivalence class of $f$.
Therefore, Question~\ref{Q:reel} is equivalent to asking when
does an antiholomorphic parabolic germ $f$ preserve a real analytic curve.

Like the embedding problem, to preserve a germ
of real axis is a condition of codimension infinity,
but it is a ``smaller'' infinity, i.e.\null{} not
every transition functions needs to be a translation,
but we will see that there are infinitely many conditions of the form
$c_{n,j} = \overline{c_{-n,-j}}$ for the Fourier coefficients
in~\eqref{eq:Psi fourier}.

\begin{theo}\label{theo:f reel}
  Let $f\colon (\Cp,0)\to (\Cp,0)$ be an antiholomorphic
  parabolic germ and $(k,b,[\Psi_1,\ldots,\Psi_k])$ be
  its modulus of classification.
  The following	statements are equivalent
  \begin{enumerate}
    \item $f$ preserves a real analytic curve at the
    origin;\vskip2pt
    \item $f$ is analytically conjugate to a germ
    with real coefficients;\vskip2pt
    \item each representative $(\Psi_1,\ldots,\Psi_k)$
    satisfies $\Psi_j\circ \Tt = \Tt \circ \Psi_j$;\vskip2pt
    \item each representative $(\Psi_1,\ldots,\Psi_k)$ satisfies
    $\Psi_j\circ \Sigma = \Sigma\circ \Psi_{-j}$,
    where $\Psi_{-j}$ is defined by~\eqref{eq:Psi moins}.
  \end{enumerate}
\end{theo}
\begin{proof}
  \noindent1.\null{} $\Rightarrow$ 2. Let $\gamma$ be
  the germ of real analytic parametrization of the invariant
  curve of $f$. We can extend $\gamma$ on a disk around the
  origin. Then $\gamma\inv\circ f\circ \gamma$ fixes
  a germ of the real axis, so its power series has
  real coefficients.

  \smallskip
  \noindent 2.\null{} $\Rightarrow$ 3.\null{} 
  We can of course suppose that the power series
  of $f$ has real coefficients, so that $\sigma\circ f = f\circ \sigma$. 
  It follows that $\sigma\circ f$ is a holomorphic square root of 
  $g = f\circ f$. Therefore, we have $\Psi_j\circ \Tt = \Tt \circ \Psi_j$.
  
  \smallskip
  \noindent3.\ $\Leftrightarrow$ 4. It follows 
  from equation~\eqref{eq:Psi commute}.

  \smallskip
  \noindent4.\ $\Rightarrow$ 1. This is the harder part 
  of the proof. The hypothesis implies that $\Sigma$
  is well defined on the Riemann surface $S$ constructed
  in the proof of part 2 of the Classification Theorem~\ref{theo:class}.
  Let $z = H(W)$ be the coordinate given by $H\colon S\to (\Cp,0)\setminus\{0\}$,
  the biholomorphism of $S$ to a punctured neighbourhood
  of the origin. Let $\sigma'$ and $f'$ be the expression
  of $\Sigma$ and $\STt$ in this coordinate (note that
  $f$ and $f'$ are analytically conjugate).
  Since on $S$, $\Sigma$ and $\Sigma\circ\Tt$ commute,
  it follows that $\sigma'$ and $f'$ commute also.
  So it remains only to show that $\sigma'$ preserves
  a germ of real analytic curve at the origin.

  We can extend $\sigma'$ by $\sigma'(0) = 0$ by
  the Riemann Removable Singularity Theorem (this is true
  for antiholomorphic functions, since we simply apply
  it to $\sigma'\circ \sigma$, where $\sigma(z) = \zbar$).
  In the charts $U_0$ and $U_k$, $\Sigma$ fixes
  the real axis. These two curves
  carry in the $z$-coordinate and meet at the
  origin to form a continuous curve $\gamma$ fixed by $\sigma'$.
  In fact, $\gamma$ is a $C^1$ curve, as $H$ can be
  extended to a $C^1$ diffeomorphism $\widetilde H\colon \widetilde S\to (\Cp,0)$,
  where $\widetilde S$ is obtained from $S$ with
  the point $\infty = H\inv(0)$ added.

  This curve divides a small disk $D(0,\delta)$ in two
  connected components $A$ and $B$. By the Riemann
  Mapping Theorem, there exists a biholomorphism $\varphi$
  of $A$ to the upper half-plane that sends continuously
  the boundary of $A$ on the real line, see Figure~\ref{fig:prolonge}.
  The image of  $\gamma$ corresponds to an interval $[a,b]$. We
  can extend $\varphi$ to $A\cup\gamma\cup B$ by 
  $$
    \widetilde\varphi(z) =\begin{cases}
	\varphi(z) & \hbox{if } z\in A\cup\gamma;\\
	\sigma\circ\varphi\circ\sigma'(z) & \hbox{if } z\in B.
    \end{cases}
  $$
  This is holomorphic on $B$, since it is the composition
  of a holomorphic map with two antiholomorphic maps, and
  it is continuous on $A\cup \gamma\cup B$. 
  
  \begin{figure}[htbp]
    \includegraphics{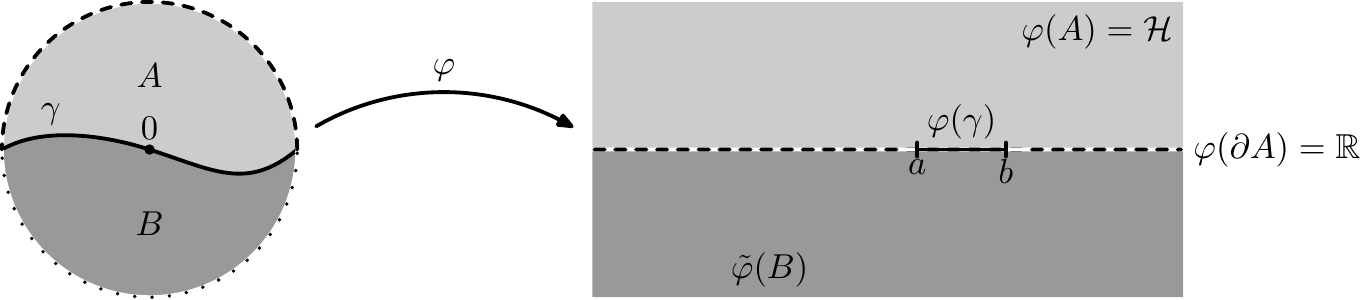}
    \caption{Mapping $\varphi$ extended on $A\cup\gamma\cup B$}
    \label{fig:prolonge}
  \end{figure}
  
  The argument
  to prove that $\widetilde\varphi$ is holomorphic is
  similar to that of the Schwarz Reflection Principle.
  The idea is to show that the integral of $\widetilde\varphi$
  along any triangle in $A\cup \gamma\cup B$ vanishes,
  and it will follow from Morera's theorem. If a simple closed curve
  is in $A\cup\gamma$ or $B\cup \gamma$, then it follows
  from Cauchy's theorem (we may take a limit of closed curves in $A$ or $B$
  converging to the initial one). Then, we can divide any
  triangle in $A\cup\gamma\cup B$ along $\gamma$ to obtain
  a finite number of closed curves in $A\cup \gamma$ and in $B\cup\gamma$,
  as in Figure~\ref{fig:tri}.
  Thus, $\widetilde\varphi\big|_{(a,b)}\inv$ is a
  real analytic parametrization of $\gamma$.
  \begin{figure}[htbp]
    \includegraphics{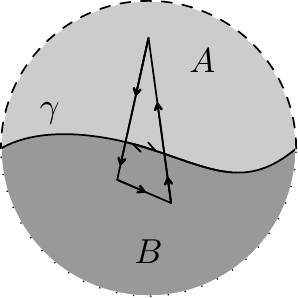}
    \caption{Triangle divided in two closed curves}
    \label{fig:tri}
  \end{figure}
\end{proof}

When $f$ is conjugate to a germ with real coefficients,
then so is the holomorphic germ $g = f\circ f$.
However, it is not true that every holomorphic germ
analytically conjugate to a germ with real coefficients
must have a germ $f$ such that $g = f\circ f$, as the next theorem will show
(see also Proposition~\ref{prop:g sans racine}).
Even though this is not a property directly linked to
the antiholomorphic parabolic germs, we will still prove
the following necessary and sufficient condition
for a holomorphic parabolic germ to preserve a
real analytic curve.

\begin{theo}\label{theo:g reel}
  Let $g\colon(\Cp,0) \to (\Cp,0)$ be a parabolic holomorphic
  germ and let $(k,b,[\Psi_{-k},\ldots,\Psi_{-1},\Psi_1,\ldots\Psi_k])$
  be its Écalle-Voronin modulus (see Section~\ref{rem:module Ecalle}).
  The following statements are equivalent
  \begin{enumerate}
    \item $g$ preserves a germ of real analytic curve 
    at the origin;
    \item $g$ is analytically conjugate to a germ 
    with real coefficients;
    \item For each representative $(\Psi_{-k},\ldots,\Psi_{-1},\Psi_1,\ldots,\Psi_k)$,
    there exists $y\in\R$ such that the transition functions
    satisfy $\Sigma\circ T_{iy}\circ \Psi_{j} = \Psi_{-j}\circ \Sigma\circ T_{iy}$;
    \item There exists a representative 
    $(\Psi_{-k},\ldots,\Psi_{-1},\Psi_1,\ldots,\Psi_k)$
    such that 
    $\Psi_{j}\circ \Sigma = \Sigma\circ \Psi_{-j}$.
  \end{enumerate}
\end{theo}
\begin{proof}
  1.\null{} $\Rightarrow$ 2.\null{} and 4.\null{} $\Rightarrow$ 1.\null{} 
  are the same as in the previous theorem.

  2.\null{} $\Rightarrow$ 3. Suppose $g$ is in a coordinate
  such that $g = \sigma\circ g\circ \sigma$. Let $\Phi_j$
  be a Fatou coordinate on $U_j$ for $-k \leq j\leq k$,
  where $\Phi_{-k} = \Phi_k$. Then $\Sigma\circ \Phi_{-j}\circ \Sigma$
  is also a Fatou coordinate of $g$ on $U_j$. By the uniqueness
  of the Fatou coordinate (for the holomorphic case), there
  exists a constant $C\in \Cp$ such that 
  $$
    \Sigma\circ \Phi_{-j}\circ\Sigma\circ \Phi_j\inv = T_C
  $$
  for all $j$. In particular, for $j=0$, by taking the
  inverse and conjugating both sides by $\Sigma$ of the previous equation, we obtain
  $T_C = T_{-\overline C}$, so $C$ must be pure imaginary, say $iy$ with $y\in\R$.
  For $j>0$ and odd, we conclude with\footnotesize
  $$
    \Psi_j = \Phi_j\circ \Phi_{j-1}\inv = T_{-iy}\circ \Sigma \circ \Phi_{-j}
      \circ \Phi_{-j+1}\inv\circ \Sigma\circ T_{iy}
      = T_{-iy}\circ \Sigma \circ \Psi_{-j} \circ \Sigma\circ T_{iy}.
  $$\normalsize
  The other values of $j$ are done similarly.

  \noindent3.\null{} $\Leftrightarrow$ 4. It follows from
  the fact that for any two representatives 
  $(\Psi_{-k},\ldots,\Psi_{-1},\Psi_1,\ldots,\Psi_k)$
  and $(\Psi_{-k}',\ldots,\Psi_{-1}',\Psi_1',\ldots,\Psi_k')$,
  there exists $C\in\Cp$ such that $\Psi_j\circ T_C = T_C \circ \Psi_j'$.
  We choose $C = -{iy\over 2}$ to get a representative
  satisfying $\Psi_j'\circ \Sigma = \Sigma\circ \Psi_{-j}'$.
\end{proof}

\begin{prop}\label{prop:g sans racine}
  There exists $g$ with real coefficients
  that has no antiholomorphic square root 
  (see section~\ref{sec:racine}).
\end{prop}
\begin{proof} 
  The holomorphic germ realized by the Écalle-Voronin
  modulus $(1,0,[W + e^{2i\pi W}, W + e^{-2i\pi W}])$
  has no antiholomorphic square root, but preserves a germ of real analytic curve.
\end{proof}

\subsection{Centralizer in the Group of Holomorphic and Antiholomorphic Germs}
Let $\Diff_1(0,\Cp)$ (resp.\null{} $\overline{\Diff_1(0,\Cp)}$)
be the set of germs of holomorphic (resp.\null{} antiholomorphic)
diffeomorphisms with a fixed point at the origin with multiplier 1.
Note in particular that an antiholomorphic germ must be tangent to
$\sigma$. We set
$$
  \DiffC = \Diff_1(0,\Cp)\cup \overline{\Diff_1(0,\Cp)}.
$$
It forms a group with $\Diff_1(0,\Cp)$ as a subgroup.
Next, let $\joli A_{k,b}\subset\Diff_1(0,\Cp)$ (resp.\null{} 
$\overline{\joli A_{k,b}}\subset\overline{\Diff_1(0,\Cp)}$ when $b$ is real)
be the set of germs of holomorphic (resp.\null{} antiholomorphic)
diffeomorphisms with a parabolic fixed point of codimension~$k$ at the origin
and with formal invariant $b$. 

We will study the centralizer of $g\in \joli A_{k,b}$ and
of $f\in \overline{\joli A_{k,b}}$ in $\DiffC$. 
Let us note the following: if $g\in\joli A_{k,b}$ commutes
with $f\in\overline{\Diff_1(0,\Cp)}$, then the formal 
invariant $b$ of $g$ is automatically real. In fact, 
we have that $g$ and $\sigma\circ g\circ \sigma$ are
analytically conjugate by means of $(\sigma\circ f)$.
So we are interested only in $b$ real, since if $b$
is not real, the centralizer of $g$ in $\DiffC$ is
the same as the centralizer of $g$ in $\Diff(0,\Cp)$,
which is already known (see \cite{nonlinear}).

\begin{lem}\label{lem:h Fatou}
  Let $g\in \joli A_{k,b}$ and let $m\in\DiffC$ be a germ that commutes with
  $g\in\joli A_{k,b}$. Then either $m$ is the identity, 
  or analytically conjugate to $\sigma$, or $m\in\AAkb$.
  Moreover, in the Fatou coordinates of $g$, $m$
  will be of the form
  \begin{equation}\label{eq:alternative h}
    T_C \qquad\hbox{ or }\qquad \Sigma\circ T_C
    \qquad\rlap{\qquad for $C\in\Cp$.}
  \end{equation}
\end{lem}
\begin{proof}
  If $m$ is not the identity or analytically conjugate
  to $\sigma$, then $m$ is parabolic. It must have 
  codimension $k$, since it maps the orbits of $g$
  on the orbits of $g$ and the Fatou petals of $g$
  on the Fatou petals of $g$. To see that $m$ has the
  same formal invariant, we compare degree by degree the 
  power series on both sides of the equation
  $m\circ g(z) = g\circ m(z)$.
  
  To see $m$ has one of the forms of~\eqref{eq:alternative h},
  the first part of the proof of Proposition~\ref{prop:coord 
  Fatou} applies almost identically to $m$.
\end{proof}

The obvious germs in the centralizer of $f$ or $g$ are the
iterates, the roots, and the iterates of the roots, which
we define below as the \emph{fractional iterates}. We 
prove in Theorem~\ref{theo:centr g} and~\ref{theo:centr f} that these
are all the possible elements of the centralizer, provided 
that $f$ or $g$ are not conjugate to their respective normal form.

\begin{deff}
  Let $f\in\overline{\joli A_{k,b}}$ (resp.\null{} $g\in\joli A_{k,b}$).
  We say that $m\in\DiffC$ is a
  \emph{fractional iterate of order $p$} of $f$ (resp.\null{} $g$)
  if there exists $q\in\Z$ such that $\gcd(p,q)=1$ and
  $m^{\circ p} = f^{\circ q}$ (resp.\null{} $m^{\circ p} = g^{\circ q}$).
\end{deff}

Note that a fractional iterate of order one is just an 
iterate of $f$ or of $g$.

\begin{theo}\label{theo:centr g}
  Let $g\in\joli A_{k,b}$ with $b\in\R$.
  Then we have one of the following cases:
  \begin{enumerate}
    \item $g$ is embeddable, i.e.\null{} $g = h\circ \vun\circ h\inv$.
    Let $g_t = h\circ v^t\circ h\inv$. Then the centralizer of $g$ is
    $$
      Z_{g} = \{g_t \mid t\in\Cp\}\cup \{h\circ \sigma\circ v^t\circ h\inv\mid t\in\Cp\}.
    $$
    \item$g$ is not embeddable, then $Z_g$ contains only
    holomorphic and antiholomorphic fractional iterates of $g$ and
    Schwarz reflections. More precisely, there exists $p\in\N\setminus\{0\}$ 
    such that the centralizer of $g$ is one of the following:\vskip2\jot
    \begin{itemize}\bgroup\let\vcenter=\vtop
    \item $
      \begin{aligned}
        Z_g = \bigcup_{d\mid p}\,\{\hbox{fractional iterates of order $d$ of $g$}\},
      \end{aligned}
    $\vskip2\jot
    \item $
      \begin{aligned}
        Z_g &= \bigcup_{d\mid p}
	  \{\hbox{hol.\null{} fractional iterates of order $d$ of $g$}\} \cup{}\\[-3\jot]
          &\qquad\qquad
	    \bigcup_{d\mid 2p\atop d\,{\it even}}
	    \{\hbox{antihol.\null{} fractional iterates of order $d$ of $g$}\},
      \end{aligned}
    $\vskip2\jot
    \item $
      \begin{aligned}
        Z_g &= \bigcup_{d\mid p}
	\{\hbox{hol.\null{} fractional iterates of order $d$ of $g$}\} \cup{}\\[-2\jot]
          &\qquad\qquad\bigcup_{d\mid p\atop d\,{\it even}}
	    \{\hbox{antihol.\null{} fractional iterates of order $d$ of $g$}\}\\[-2\jot]
	  &\qquad\qquad\qquad\qquad\cup\{\hbox{Schwarz reflection tangent to $\sigma$}\}.
      \end{aligned}
    $\vskip2\jot
    \egroup\end{itemize}\vskip2\jot
    The centralizer includes a Schwarz reflection tangent 
    to $\sigma$ if and only if $g$ is analytically conjugate to 
    a germ with real coefficients (see Theorem~\ref{theo:g reel}).
  \end{enumerate}
\end{theo}
\begin{proof}
  \noindent 1.\null{}
  Since $g$ is embeddable, each of its transition functions
  is a translation, so any $T_t$ will commute with them.
  Each $T_t$ represents a germ analytically conjugate
  to $g^t$ for some $t\in\Cp$.
  In the case $m$ is antiholomorphic, $\Sigma\circ T_t$ is compatible with
  the transition functions with no restriction on $t$. It corresponds to 
  $h\circ \sigma\circ v^{t}\circ h\inv$ for $t\in\Cp$.

  \smallbreak
  \noindent 2.\null{} Since $g$ is not embeddable, 
  one of the transition functions is not a 
  translation. Suppose $g_1\in \DiffC$ is holomorphic and
  commutes with $g$. By Lemma~\ref{lem:h Fatou}, $g_1$
  becomes $T_t$ in the Fatou coordinates for some $t\in\Cp$.
  Then $T_t$ must  commute with the transition
  functions. This will only happen if
  the transition functions expand in the form
  \begin{equation}\label{eq:Psi fourier 2}
    \Psi_j(W) = W + (-1)^j{i\pi b\over k}
      + \sum_{\ell\in\Z^\ast} c_{\ell p,j}e^{2i\pi \ell p W}
  \end{equation}
  and $t  = {a\over p}$ for some $a,p\in\N$.
  This corresponds to $g_1$ being a 
  holomorphic fractional iterate of order $p$ of $g$.
  Let $p$ denote the maximal order of holomorphic
  roots of $g$.

  Suppose $f_1\in\DiffC$ is antiholomorphic and commutes
  with $g$. We can suppose $f_1$ is not a Schwarz reflection, 
  since this is covered in Theorem~\ref{theo:g reel}. Then by
  Lemma~\ref{lem:h Fatou}, $f_1$ is $\STtt_t$ in the Fatou coordinates 
  and it must be compatible with the transition functions. Therefore,
  $T_{2\Re t}$ commutes with the $\Psi_j$'s, and by the
  previous case, we have $2\Re t = {a \over r}$, with
  $\gcd(a,r) = 1$ and $r\,|\, p$. We have 
  $\STtt_t\circ \Psi_j = \Psi_{-j}\circ\STtt_t$
  for some representative of the modulus
  if and only if
  \begin{equation}\label{eq:coeff fourier commute}
    \overline{c_{n,j}} = e^{-i\pi an/r}e^{2n\pi y}c_{-n,-j}
  \end{equation}
  and $t = {a\over 2r} + iy$ ($y\in\R$), where $c_{n,j}$ 
  are the Fourier coefficients of $\Psi_j$ 
  from~\eqref{eq:Psi fourier}. Of course, generically there exists
  no $y$ that satisfy~\eqref{eq:coeff fourier commute}, so 
  there are no antiholomorphic germs that commute with $g$ in 
  the generic case. 
  When such a $y$ exists, by changing the $\Psi_j$'s 
  by $T_{-{iy\over2}}\circ \Psi_j\circ T_{{iy\over2}}$, we can suppose
  $y=0$.

  Since $\Ttt_{1\over p}$ commutes with the $\Psi_j$'s, 
  we have $c_{n,j} = 0$ if \hbox{$p\!\!\not|\,n$},
  so that~\eqref{eq:coeff fourier commute} becomes 
  $\overline{c_{sp,j}} = e^{-i\pi s ap/r} c_{-sp,-j}$.

  If $p/r$ and $a$ are odd, then~\eqref{eq:coeff fourier commute} 
  becomes
  $$
    \overline{c_{sp,j}} = \begin{cases}
      -c_{sp,j}, & \hbox{$s$ odd,}\\
      c_{sp,j}, & \hbox{$s$ even.}
      \end{cases}
  $$
  This is precisely the condition for $g$ to have an
  antiholomorphic root of order $2p$.

  If $p/r$ or $a$ is even, then $e^{-i\pi asp/r} = 1$ for all $s$,
  so we obtain the condition necessary for $\Sigma$
  to be compatible with the transition functions, and then
  there is a Schwarz reflection in the centralizer of $g$. In
  that case, then the highest orders
  of the holomorphic and antiholomorphic iterates coincide.
\end{proof}

Lastly, we will study the centralizer of $f\in\overline{\joli A_{k,b}}$
in $\DiffC$. Of course, we have $Z_f \subseteq Z_{f\circ f}$, so the
details are similar to the previous theorem.

\penalty-500
\begin{theo}\label{theo:centr f}
  Let $f\in\overline{\joli A_{k,b}}$. Then we have one of
  the following cases:
  \begin{enumerate}
    \item $f$ is embeddable, i.e.\null{} $f = h\circ \sigma\circ\vt\circ h\inv$,
    then its centralizer is
    $$
     \kern\parindent Z_{f} = \{h\circ v^t\circ h\inv\,|\, t\in\R\}
        \cup \{h\circ \sigma\circ v^{t}\circ h\inv\,|\, t\in\R\}.
    $$
    \item $f$ is not embeddable, then $Z_f$ contains only
    holomorphic and antiholomorphic fractional iterates of $f$ and
    Schwarz reflections. More precisely, there exists $p\in\N\setminus\{0\}$ 
    such that the centralizer of $f$ is one of the following:\vskip2\jot\penalty-500
    \begin{itemize}
    \item if $p$ is odd, then \vadjust{\vskip2\jot}\newline
    $
      \begin{aligned}
        Z_f ={} &\bigcup_{d\mid p}
	    \{\hbox{hol.\null{} fractional iterates of order $d$ of $f\circ f$}\}\cup{}\\
          &\qquad\bigcup_{d\mid p\atop d\,{\it odd}}
	    \{\hbox{antihol.\null{} fractional iterates of order $d$ of $f$}\};\\
      \end{aligned}
    $\vskip2\jot
    \item if $p$ is even, then \vadjust{\penalty5000\vskip2\jot}\newline
    $\hskip2\parindent
      \begin{aligned}
        Z_f ={} &\bigcup_{d\mid p}
	    \{\hbox{hol.\null{} fractional iterates of order $d$ of $f\circ f$}\}\cup{}\\
          &\qquad\bigcup_{d\mid p\atop d\,{\it odd}}
	    \{\hbox{antihol.\null{} fractional iterates of order $d$ of $f$}\};\\
	  &\qquad\qquad\cup\{\hbox{Schwarz reflection tangent to $\sigma$}\}.\\
      \end{aligned}
    $
    \end{itemize}\vskip2\jot
    The centralizer includes a Schwarz reflection tangent 
    to $\sigma$ if and only if $f$ is analytically conjugate to 
    a germ with real coefficients (see Theorem~\ref{theo:f reel}).
  \end{enumerate}
\end{theo}
\begin{proof}
  \noindent1.\null{} In the Fatou coordinates, $T_t$ must
  commute with $\STt$, so $t$ must be real. The rest of
  the details are exactly as in the proof of Theorem~\ref{theo:centr g}.

  \smallskip\noindent2.\null{}
  Let $p\in \N$ be the highest order of the holomorphic
  roots of $f\circ f$. This number might be $1$, in which case
  $Z_f$ contains only iterates of $f$. Indeed, if $p=1$, $Z_f$
  does not contain a  Schwarz reflection, otherwise $f\circ f$ would have the
  holomorphic root $\sigma\circ f$, which would contradict that
  the highest order of the holomorphic root is $1$.
  Let us suppose that $p>1$.

  Let $(\Psi_1,\ldots,\Psi_k)$ be a representative of the analytic
  invariant of $f$. Let $\Psi_{-j}$ be determined by~\eqref{eq:Psi moins}
  for $j=1,\ldots,k$. Suppose that the transition functions have a
  Fourier expansion
  $$
    \Psi_j(W) = W + C_j + \sum_{n}
      c_{n,j}e^{2i\pi nW},
  $$
  for all $j$, where $\overline{C_{-j}} = C_j = (-1)^j{i\pi b\over k}$ for $j>0$. 
  Because the $\Psi_j$'s satisfy $\Psi_j\circ\STt = \STt \circ \Psi_{-j}$,
  we obtain
  \begin{equation}\label{eq:coeff bar}
    \overline{c_{n,j}} = e^{i\pi n} c_{-n,-j}.
  \end{equation}
  Since $f\circ f$ has a root of order $p$, we also have
  $T_{1\over p}\circ \Psi_j = \Psi_j\circ T_{1\over p}$,
  which means that $c_{n,j} = 0$ for all $j$ when \hbox{$p\!\!\not|\, n$}.
  If we write $n = \ell p$ in~\eqref{eq:coeff bar}, it becomes
  $$
    \overline{c_{\ell p,j}} = e^{i\pi \ell p} c_{-\ell p,-j},
    \rlap{\qquad $(\ell\in\Z)$.}
  $$
  We have two cases.
  
  If $p$ is odd, then we have
  $$
    \overline{c_{\ell p,j}} = 
    \begin{cases}
      c_{-\ell p,-j} & \hbox{ if $\ell$ is odd,}\\
      -c_{-\ell p,-j} & \hbox{ if $\ell$ is even.}
    \end{cases}
  $$
  This corresponds to the necessary condition for
  $f$ to have an antiholomorphic $p$-th root.
  Therefore, the maximal order of antiholomorphic
  root of $f$ is at least $p$. To see that it is exactly $p$,
  we simply note that if $f_1$ is an antiholomorphic root
  of order $q \geq p$, then $f_1\circ f_1$ is a holomorphic
  root of order $q$ of $f\circ f$, and because $p$
  is maximal, we must have $p=q$. 
  We must also prove that antiholomorphic fractional iterates
  of $f\circ f$ are fractional iterates of $f$. Suppose
  $f\circ f$ has an antiholomorphic root of order $q$,
  where $q$ must be even.
%
  Then we have ${q\over 2}\,|\, p$. Since $p$
  is odd, we know that $q\over 2$ must be
  odd. It follows that $(\STtt_{1\over q})^{\circ {q\over 2}} = \STt$.

  If $p$ is even, then we have $\overline{c_{n,j}} = c_{-n,-j}$,
  which corresponds to the condition necessary for $\Sigma$
  to be compatible with the $\Psi_j$'s. In this case,
  the centralizer of $f$ contains a Schwarz reflection,
  and it follows that the highest order of antiholomorphic
  fractional iterates of $f$ and the highest order of holomorphic
  fractional iterates of $f\circ f$ coincide.

%
\end{proof}

\color{black}
\bigbreak
\bibliographystyle{plain}
\bibliography{ref}

\end{document}